\documentclass[12pt]{amsart}
\usepackage{amstext,amsfonts,amssymb,amscd,amsbsy,amsmath,verbatim,tikz}
\usepackage{ifthen,fullpage}
\usepackage{color,tikz}
\usepackage{amsthm}
\usepackage{latexsym}
\usepackage[all]{xy}
\usepackage{enumerate}

\newtheorem{lemma}{Lemma}[section]
\newtheorem{theorem}[lemma]{Theorem}

\newtheorem{prop}[lemma]{Proposition}
\newtheorem{cor}[lemma]{Corollary}

\newtheorem{claim*}{Claim}
\newtheorem{thm}[lemma]{Theorem}
\newtheorem{defn}[lemma]{Definition}
\newtheorem{example}[lemma]{Example}

\theoremstyle{remark}
\newtheorem{remark}[lemma]{Remark}

\newcommand{\initial}{\operatorname{in}}

\newcommand{\Spec}{\operatorname{Spec}}

\newcommand{\Div}{\operatorname{Div}}

\newcommand{\Tor}{\operatorname{Tor}}

\newcommand{\Hom}{\operatorname{Hom}} %done

\newcommand{\Gr}{\operatorname{Gr}}

\newcommand{\rank}{\operatorname{rank}}
\newcommand{\codim}{\operatorname{codim}}

\newcommand{\Sym}{\operatorname{Sym}} %done
\newcommand{\GL}{{GL}}

\newcommand{\defi}[1]{\textsf{#1}} % for defined terms

\newcommand{\chr}{\ensuremath{\operatorname{char}}}

\title{A syzygetic approach to the smoothability of zero-dimensional schemes}
\author{Daniel Erman}
\address{Department of Mathematics, University of California,
	Berkeley, CA 94720, USA}
\email{derman@math.berkeley.edu}
\urladdr{http://math.berkeley.edu/\~{}derman}

\author{Mauricio Velasco}
\address{Department of Mathematics, University of California,
	Berkeley, CA 94720, USA}
\email{velasco@math.berkeley.edu}
\urladdr{http://math.berkeley.edu/\~{}velasco}
\thanks{Daniel Erman is partially supported by an NDSEG grant.  Mauricio Velasco is partially supported by NSF grant DMS-0802851.}

\begin{document}
\maketitle

\begin{abstract}
We consider the question of which zero-dimensional schemes deform to a collection of distinct points; equivalently, we ask which Artinian $k$-algebras deform to a product of fields.  We introduce a syzygetic invariant which sheds light on this question for zero-dimensional schemes of regularity two.  This invariant imposes obstructions for smoothability in general, and it completely answers the question of smoothability for certain zero-dimensional schemes of low degree.  The tools of this paper also lead to other results about Hilbert schemes of points, including a characterization of nonsmoothable zero-dimensional schemes of minimal degree in every embedding dimension $d\geq 4$.
\end{abstract}

\section{Introduction}
A fundamental question in the study of zero-dimensional schemes is to determine which $0$-schemes deform to a collection of distinct points\footnote{An equivalent question is to determine which Artinian $k$-algebras deform to $k^n$.}, that is, which $0$-schemes are smoothable  (cf.\  \cite{iarrobino-red}, \cite{iarrobino-small-tangent-space},\cite{fogarty}, \cite{mazzola}, \cite{shafarevich}, \cite{evain}, \cite{cartwright}). For embedding dimension greater than two, very little is known about how to answer this question.  In this paper, we introduce a syzygetic invariant which yields new and remarkably sharp information about this question.  Our invariant imposes necessary conditions for smoothability of $0$-schemes of regularity two, and it completely determines the question of smoothability in low degree.

Previous work on smoothability focuses on tangent space dimension.  Since the dimension of the first order deformation space of a $0$-scheme $\Gamma\subseteq \mathbb A^d$ is upper semicontinuous,  having a ``small tangent space'' poses an obstruction to smoothability.  This notion of a ``small tangent space'' obstruction is introduced and exploited in \cite{iarrobino-small-tangent-space}, where the graded structure of the tangent space is also used to show that a generic homogeneous ideal with Hilbert function $(1,4,3)$ is nonsmoothable.  Shafarevich greatly expanded on these results by a similar ``small tangent space'' obstruction in~\cite{shafarevich}.

Despite the significant results of \cite{iarrobino-small-tangent-space} and \cite{shafarevich}, tangent space dimension is a rather coarse invariant in the study of smoothability.  There exist many possible causes for an increase in the number of first order deformations, and these are not necessarily related to smoothability.  For instance, if a $0$-scheme belongs to the intersection of two irreducible components of the Hilbert scheme, then this $0$-scheme will have a large deformation space, but it may not be smoothable (cf.\   Example~\ref{examples} \eqref{ex:intersect}).

The invariant introduced below imposes obstructions to smoothability for homogeneous $0$-schemes of regularity two and, in some cases, provides even richer information.  For instance, our completely answers the question of smoothability for certain $0$-schemes of low degree (cf.\  Theorem~\ref{thm:genwaldo}).

\subsection{The $\kappa$-vector and obstructions to smoothability}
The invariant introduced in this paper is called the \defi{$\kappa$-vector} of a homogeneous ideal.   We work over an algebraically closed field $k$ with $\chr(k)\ne 2,3$.  We say that a $0$-scheme $\Gamma$ has regularity two if $H^0(\Gamma,\mathcal O_\Gamma)$ is a local ring whose maximal ideal $\mathfrak m$ satisfies $\mathfrak m^3=0$.  In the case that $H^0(\Gamma,\mathcal O_\Gamma)$ is a graded quotient of the standard graded polynomial ring, this notion of regularity coincides with the familiar notion of Castelnuovo-Mumford regularity.  Note that every $0$-scheme of regularity two is irreducible but not reduced.  Since $0$-schemes of regularity two are one of the simplest classes containing infinitely many distinct isomorphism types (cf. \cite[p.\ 1335]{shafarevich} or \cite[Lemma 1.2(3)]{poonen}), it is natural to focus on these families.

Every $0$-scheme of regularity two and embedding dimension $d$ admits an embedding $\Gamma \subseteq \mathbb A^d$ such that $\Gamma$ is represented by a homogeneous ideal $I\subseteq S:=k[x_1, \dots, x_d]$.  Conversely, every embedding of $\Gamma$ into $\mathbb A^d$ is, up to translation, defined by a homogeneous ideal $I$.  Note that the deformation theory of an embedded $0$-scheme is smooth over the abstract deformation theory of the $0$-scheme~\cite[p.\ 4]{artin}. Hence we may fix such an embedding of $\Gamma$ without affecting its deformation theoretic properties.  Let $e=\deg(\Gamma)-d-1$ and let $I_2^\perp \in \Gr(e,S_2^*)$ be the degree two part of the (Macaulay) inverse system of the ideal $I$.  Choose a basis $q_1, \dots, q_e$ of $I_2^{\perp}$ and represent these elements by $d\times d$-symmetric matrices $A_1, \dots, A_e$.  We define $\kappa_j(I)$ to be the rank of the following linear map induced by $\mathbf{A}=(A_1, \dots, A_e)\in S_2^*\otimes I_2^\perp$ (see \S\ref{sec:Kvectors} for a more detailed discussion and for an equivalent definition via syzygies):
\[ \xymatrix{
S_1\otimes \bigwedge^{j}\left( I_2^{\perp}\right) \ar[r]^{\wedge \mathbf{A} \quad} & S_1^*\otimes \bigwedge^{j+1}\left( I_2^{\perp}\right).\\
}
\]
More concretely, when $e=3$, the numbers $\kappa_0(I), \kappa_1(I), \kappa_2(I)$ are the ranks of the matrices appearing in the following sequence:
$$
\xymatrix@+3pc{
k^d \ar[r]^{\left(\begin{smallmatrix} A_1\\A_2\\A_3 \end{smallmatrix}\right)} &
k^{3d} \ar[r]^{\left(\begin{smallmatrix} 0&A_3&-A_2\\ -A_3 & 0 & A_1 \\ A_2&-A_1&0 \end{smallmatrix}\right)} &
k^{3d} \ar[r]^{\left(\begin{smallmatrix} A_1&A_2&A_3 \end{smallmatrix}\right)} &
k^d.
}
$$

Variations of $\kappa_1$ have appeared in several geometric settings.  Namely, when $e=3$ the invariant $\kappa_1$ is determined by the Strassen equation, and it was previously studied in connection with secant varieties~\cite{ottaviani},~\cite[p.\ 14]{abo}, vector bundles~\cite{ottaviani}, and polynomial versions of Waring's problem~\cite[p.\ 513]{carlini}.

\begin{defn}
The $\kappa$-vector of $I$ is the sequence $\kappa(I)=(\kappa_0(I), \dots, \kappa_{e-1}(I))$.
\end{defn}
The $\kappa$-vector of $I$ is independent of the choice of basis of $I_2^\perp$ and is invariant under the $GL(d)$-action on $S_2^*$.  Further, the action of $GL(d)$ is transitive on the homogeneous ideals $I$ which define some embedding of $\Gamma \subseteq \mathbb A^d$. Hence each $\kappa_j(I)$ is in fact an invariant of $\Gamma$ itself, and we thus refer to $\kappa_j(\Gamma)$ and $\kappa_j(I)$ interchangeably.  The lower semincontinuity of $\kappa_j$ induces obstructions to the existence of deformations among algebras.

In Proposition \ref{prop:genKvec}, we compute the values of the $\kappa$-vector for generic ideals and generic smoothable ideals with a given embedding dimension and degree.  This computation motivates the introduction of the following two conditions:
\begin{equation*}\label{eqn:nec:1}
\kappa_j(\Gamma)\leq (d+e)\binom{e-1}{j}\tag{*}
\end{equation*}
for $j=1, \dots, e-1$ and
\begin{equation*}\label{eqn:nec:j}
\kappa_1(\Gamma)\leq (e-1)d+\binom{e}{2}.\tag{**}
\end{equation*}

\begin{thm}\label{thm:1denonsmoothable}
Let $\Gamma\subseteq \mathbb A^d$ be a $0$-scheme of regularity two, embedding dimension $d$, and degree $1+d+e$.  Then \eqref{eqn:nec:1} and \eqref{eqn:nec:j} are necessary conditions for the smoothability of $\Gamma$.
\end{thm}

\noindent
The above obstructions are nontrivial when $e\geq 3$ (c.f Example~\ref{examples}~\eqref{ex:family}).

\begin{figure}\label{fig:5125}
\begin{center}
\begin{tikzpicture}
\draw (0,1.5) node {If $\kappa(\Gamma)\leq (5,12,5),$ then};
\draw (0,1) node {the $0$-scheme $\Gamma$ \dots};
\filldraw (0,0) circle (3pt);
\draw[very thick] (-.3,-.3) -- (.3,.3);
\draw[very thick] (-.2,-.4) -- (.2,.4);
\draw[very thick] (-.1,-.45) -- (.1,.45);
\draw[very thick] (0,-.5) -- (.0,.5);
\draw[very thick] (-.3,.3) -- (.3,-.3);
\draw[very thick] (-.2,.4) -- (.2,-.4);
\draw[very thick] (-.1,.45) -- (.1,-.45);
\draw[very thick] (0,.5) -- (.0,-.5);
\draw[very thick] (.3,-.3) -- (-.3,.3);
\draw[very thick] (-.4,-.2) -- (.4,.2);
\draw[very thick] (-.45,-.1) -- (.45,.1);
\draw[very thick] (-.5,0) -- (.5,0);
\draw[very thick] (.4,-.2) -- (-.4,.2);
\draw[very thick] (.45,-.1) -- (-.45,.1);
\draw[very thick] (.5,0) -- (-.5,0);

\draw[thick] (2,0) .. controls (2.3,.3) and (2.6,.3) .. (3,0);
\draw[thick] (3,0) .. controls (3.3,-.3) and (3.6,-.3) .. (4,0);
\draw[thick, ->] (4,0) .. controls (4.3,.3) and (4.6,.3) .. (5,0);

\draw (7,1.5) node {\dots deforms to $9$ distinct points.};
\filldraw (7,0) circle (1.5pt);
\filldraw (7.2,.4) circle (1.5pt);
\filldraw (7.31,-.42) circle (1.5pt);
\filldraw (6.79,.62) circle (1.5pt);
\filldraw (6.6,.1) circle (1.5pt);
\filldraw (6.3,.2) circle (1.5pt);
\filldraw (7.33,.66) circle (1.5pt);
\filldraw (6.63,-.2) circle (1.5pt);
\filldraw (6.5,-.5) circle (1.5pt);
\end{tikzpicture}
\\
\vspace{1cm}
\begin{tikzpicture}
\draw (0,1.5) node {If $\kappa(\Gamma)\leq (4,12,4),$ then\dots };
%\draw (0,1) node {the $0$-scheme $\Gamma$ \dots};
\filldraw (0,0) circle (3pt);
\draw[very thick] (-.3,-.3) -- (.3,.3);
\draw[very thick] (-.2,-.4) -- (.2,.4);
\draw[very thick] (-.1,-.45) -- (.1,.45);
\draw[very thick] (0,-.5) -- (.0,.5);
\draw[very thick] (-.3,.3) -- (.3,-.3);
\draw[very thick] (-.2,.4) -- (.2,-.4);
\draw[very thick] (-.1,.45) -- (.1,-.45);
\draw[very thick] (0,.5) -- (.0,-.5);
\draw[very thick] (.3,-.3) -- (-.3,.3);
\draw[very thick] (-.4,-.2) -- (.4,.2);
\draw[very thick] (-.45,-.1) -- (.45,.1);
\draw[very thick] (-.5,0) -- (.5,0);
\draw[very thick] (.4,-.2) -- (-.4,.2);
\draw[very thick] (.45,-.1) -- (-.45,.1);
\draw[very thick] (.5,0) -- (-.5,0);

\draw[thick] (2,0) .. controls (2.3,.3) and (2.6,.3) .. (3,0);
\draw[thick] (3,0) .. controls (3.3,-.3) and (3.6,-.3) .. (4,0);
\draw[thick, ->] (4,0) .. controls (4.3,.3) and (4.6,.3) .. (5,0);

\draw (7,1.5) node {\dots we can peel a single};
\draw(7,1) node {point off of $\Gamma$.};
\filldraw (7,0) circle (3pt);
\draw[very thick] (6.8,-.25) -- (7.3,.25);
\draw[very thick] (6.85,-.35) -- (7.15,.35);
\draw[very thick] (6.95,-.4) -- (7.05,.4);
\draw[very thick] (7,-.45) -- (7,.45);
\draw[very thick] (6.75,.25) -- (7.25,-.25);
\draw[very thick] (6.85,.35) -- (7.15,-.35);
\draw[very thick] (6.9,.45) -- (7.1,-.45);
\draw[very thick] (7,.45) -- (7.0,-.45);
\draw[very thick] (7.25,-.25) -- (6.75,.25);
\draw[very thick] (6.65,-.15) -- (7.3,.15);
\draw[very thick] (6.6,-.05) -- (7.4,.05);
\draw[very thick] (6.55,0) -- (7.45,0);
\draw[very thick] (7.35,-.15) -- (6.65,.15);
\draw[very thick] (7.45,0) -- (6.55,0);
\filldraw (8,0) circle (1.5pt);
\end{tikzpicture}
\end{center}
\caption{In Example~\ref{examples}\eqref{ex:10pts} we consider a $0$-scheme $\Gamma$ of regularity two, embedding dimension $5$, and degree $9$.  The example illustrates how the $\kappa$-vector contains rich information about the deformations of $\Gamma$.}
\end{figure}
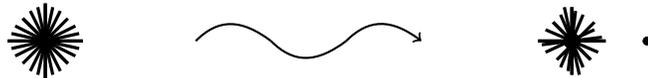

\subsection{Minimal examples and sufficient conditions for smoothability}
We next consider minimal examples of nonsmoothable $0$-schemes.  The results of \cite{cartwright} imply that every $0$-scheme of degree at most $7$ is smoothable and that every nonsmoothable $0$-scheme of degree $8$ embeds into $\mathbb A^4$ and is, up to translation, defined by a homogeneous ideal $I\subseteq k[x_1, \dots, x_4]$ with Hilbert function $(1,4,3)$.  We extend this result by characterizing, for any $d\geq 4$, the minimal degree nonsmoothable $0$-schemes of embedding dimension $d$.  The proof of the following theorem relies heavily on results from~\cite{cartwright}.

\begin{thm}\label{thm:minexamples}
For $d\geq 4,$ let $\Gamma \subseteq \mathbb A^d$ be a minimal degree subscheme of $\mathbb A^d$ which is not smoothable and which cannot be embedded in $\mathbb A^{d-1}$.  Then, up to translation, $\Gamma$ is defined by a homogeneous ideal $I\subseteq S$ with Hilbert function $(1,d,3)$.  Moreover, any sufficiently generic homogeneous ideal with Hilbert function $(1,d,3)$ defines a nonsmoothable $0$-scheme.
\end{thm}
When considering a family whose generic element is nonsmoothable, it is natural to seek sufficient conditions for the smoothability of a specific member of the family.  To the authors' knowledge, the only previously known set of nontrivial sufficient conditions for smoothability of a $0$-scheme is given in \cite[Thm.\ 1.3]{cartwright}, where the case of Hilbert function $(1,4,3)$ is considered.  By focusing on the $\kappa$-vector, we extend \cite[Thm.\ 1.3]{cartwright} to new families of $0$-schemes.
\begin{thm}\label{thm:genwaldo}
Assume that $\text{char}(k)=0$.  As in Theorem~\ref{thm:minexamples}, let $\Gamma\subseteq \mathbb A^d$ be a $0$-scheme of regularity two, embedding dimension $d$ and degree $d+4$.  In the following cases, \eqref{eqn:nec:1} and \eqref{eqn:nec:j} are sufficient conditions for the smoothability of $\Gamma$:
\begin{enumerate}
    \item  When $d\leq 8$.
    \item  When $d\leq 11$ and $I_2^\perp$ contains a nonsingular quadric.
\end{enumerate}
If $d\geq 12,$ then conditions \eqref{eqn:nec:1} and \eqref{eqn:nec:j} are not sufficient to guarantee the smoothability of $\Gamma$.
\end{thm}
%\noindent We prove parts (1) and (2) of Theorem~\ref{thm:genwaldo} without explicitly constructing smoothings.
%Instead, we isolate an irreducible component of the subset of the grassmanian $\Gr(3,S_2^*)$ with $\kappa(I)\leq (d,2d+2,d)$, and we show that this component contains all smoothable $(1,d,3)$-ideals.  By bounding the dimension of the set of smoothable $(1,d,3)$-ideals from below, and by bounding the dimension of this component from above, we then show that the two subsets coincide.

\begin{remark}
The inequalities \eqref{eqn:nec:1} and \eqref{eqn:nec:j} may be viewed as inducing determinantal equations for the intersection of components of the Hilbert scheme of points.  For instance, the set of homogeneous ideals $I\subseteq S$ with Hilbert function $(1,d,3)$ is parametrized by the Grassmanian $\Gr(3,S_2^*)$.  Theorem~\ref{thm:genwaldo} shows that, when $d\leq 8$, the determinantal equations induced by $\kappa_1$ precisely cut out the intersection of $\Gr(3,S_2^*)$ with the smoothable component of the Hilbert scheme of $d+4$ points in $\mathbb A^d$.

When $d=4$, there is an interesting comparison with the tangent space obstruction.  In this case, $\kappa_1$ determines a single equation $P$ on $\Gr(3,S_2^*)$.  As illustrated in \cite{cartwright}, the tangent space obstruction of \cite{iarrobino-small-tangent-space} also induces a determinantal equation $Q$ on $\Gr(3,S_2^*)$, and \cite[Lemma 5.17]{cartwright} shows that $P$ is irreducible and that $Q=P^8$.
\end{remark}

\subsection{Further results on Hilbert schemes of points}
The tools introduced in this paper lead to other new results about Hilbert schemes of points.  The results below are phrased for the Hilbert scheme of points in $\mathbb A^d$, but the obvious analogues also hold for the Hilbert scheme of points in $\mathbb P^d$.

\begin{prop}\label{prop:codim1}
Let $R^{d}_{n}$ be the smoothable component of the Hilbert scheme of $n$ points in $\mathbb{A}^{d}$ with $d\geq 11$ and $n\geq 15$.  There exists a closed subset $Z\subseteq R^{d}_{n}$ of codimension $1$ such that every point of $Z$ is a singular point of the Hilbert scheme.
\end{prop}

\begin{example}\label{examples} Assume that $\text{char}(k)=0$.
\begin{enumerate}
	\item\label{ex:family} ({\bf Nonsmoothable Families:})  If $e\geq 3$ and $d>\binom{e}{2}$, then a generic $0$-scheme of regularity two, embedding dimension $d$, and degree $1+d+e$, is not smoothable.
%	\item\label{ex:codim1}  ({\bf Singular in codimension $1$}):  Let $R^{d}_{n}$ be the smoothable component of the Hilbert scheme of $n$ points in $\mathbb{A}^{d}$ with $d\geq 11$ and $n\geq 15$.  There exists a closed subset $Z\subseteq R^{d}_{n}$ of codimension $1$ such that every point of $Z$ is a singular point of the Hilbert scheme.
	\item\label{ex:intersect}  ({\bf Intersections away from the smoothable component}): The Hilbert scheme of $11$ points in $\mathbb A^7$ contains two components which intersect away from the smoothable component.
	\item\label{ex:10pts}  ({\bf 9 Points in $\mathbb A^5$}):  The $\kappa$-vector yields detailed information abou the deformations of $0$-schemes of degree $9$ in $\mathbb A^5$.  We say that $J\subseteq k[x_1, \dots, x_5]$ is a  $(1,4,3)^{+1}$-ideal if it is the intersection of a homogeneous ideal with Hilbert function $(1,4,3)$ and the ideal of a reduced point.  Let $I\subseteq k[x_1, \dots, x_5]$ be a homogeneous ideal with Hilbert function $(1,5,3)$ defining a $0$-scheme $\Gamma$.  Then we have
\begin{center}
\begin{tabular}{| l | c | }\hline
$\Gamma$ deforms into a \dots & if and only if \dots \\ \hline
union of $9$ distinct points &$\kappa(I)\leq (5,12,5)$\\ \hline
$(1,4,3)^{+1}$-ideal & $\kappa(I) \leq (4,12,4)$ \\ \hline
{\em smoothable} $(1,4,3)^{+1}$-ideal & $\kappa(I)\leq (4,10,4)$\\ \hline
\end{tabular}
\end{center}
\end{enumerate}
\end{example}

\begin{remark}
Both~\cite{iarrobino-small-tangent-space} and~\cite{shafarevich} consider tangent space obstructions to the smoothability of $0$-schemes of regularity two.  The strongest result along these lines is  \cite[Thm.\ 2]{shafarevich} which implies that, if $\Gamma$ is a generic $0$-scheme satisfying the conditions of Theorem~\ref{thm:1denonsmoothable}, then $\Gamma$ is nonsmoothable whenever $2<e\leq \frac{(d-1)(d-2)}{6}+2$.  Note that Shafarevich's result is strictly stronger than our Example~\ref{examples}~\eqref{ex:family}.
\end{remark}

\subsection{Outline of paper}
The material in this paper is organized as follows. Notation and background on Hilbert schemes, inverse systems, and other topics is included in \S\ref{sec:background}.  In \S\ref{sec:smoothable1de}, we present a dominant rational map to the smoothable regularity two ideals parametrized $\Gr(e,S_2^*)$.  In \S\ref{sec:Kvectors}, we elaborate on the definition of the $\kappa$-vector of an ideal, and we compute values of the $\kappa$-vector for generic and generic smoothable ideals of regularity two.  We also introduce a module whose graded Betti numbers encode the $\kappa$-vector of an ideal. In  \S\ref{sec:Kcycles} we introduce $\kappa$-cycles, which are $\GL(d)$-equivariant subsets of $\Gr(e,S_2^*)$ defined in terms of the $\kappa$-vector, and which play a role in the proof of Theorem~\ref{thm:genwaldo}.  In \S\ref{sec:proofs}, we combine the results of the earlier sections to prove Theorems~\ref{thm:1denonsmoothable}, \ref{thm:minexamples}, and \ref{thm:genwaldo}.  Finally, in \S\ref{sec:applications}, we discuss further connections between deformations of $0$-schemes of regularity and the $\kappa$-vector, and we present the results listed in Example~\ref{examples}.

%%%%%%%%%%%%%%
%%%%%%%%%%%%%%
%%%%%%%%%%%%%%
\section{Background and Notation}\label{sec:background}
%%%%%%%%%%%%%%
%%%%%%%%%%%%%%
%%%%%%%%%%%%%%

%%%%%%%%%%%%%%
\subsection{Hilbert schemes}
%%%%%%%%%%%%%%
We use the notation $H^d_n$ for the Hilbert scheme of $n$ points in $\mathbb A^d$, and we let $R^d_n$ stand for the smoothable component of $H^d_n$.  We now discuss a coordinate system on $H^d_n$.  The reader should refer to~\cite[Ch.\ 18]{cca} for details. Given a monomial ideal $M$ of colength $n$ with standard monomials $\lambda$, let $U_{\lambda}\subseteq H^d_n$ be the set of ideals $I$ such that the monomials in $\lambda$ are a basis for $S/I$. Note that the sets $U_{\lambda}$ form an open cover of $H^d_n$.  An ideal $I\in U_{\lambda}$ has generators of the form $m-\sum_{m'\in\lambda}c_{m'}^mm'$.
The $c_{m'}^m$ are local coordinates for $U_{\lambda}$ which define a closed immersion into affine space.

%%%%%%%%%%%%%%
\subsection{Inverse systems}
%%%%%%%%%%%%%%v
Let $V$ be the vector space $V=\langle x_1, \dots, x_d\rangle$.  The symmetric algebra $\Sym^\bullet(V)$ is isomorphic to the polynomial ring $S=k[x_1,\dots, x_d]$ with the usual grading.  We define $S^*$ to be the divided power algebra $\Div^\bullet(V^*)$.  The ring $S^*$ is a graded algebra and there is a perfect pairing $S_j \times S_j^* \rightarrow S_0^* = k$.

Via this perfect pairing, it is equivalent to give a subspace $I_i\subseteq S_i$ or its orthogonal subspace $I_i^\perp\ \subseteq S_i^*$.  If $I$ is a homogeneous ideal in $S$ then we set $I^\perp = \oplus I_j^\perp$.  The space $I^\perp$ is called the (Macaulay) inverse system of the ideal $I$.  Let $y_1, \dots, y_d$ a basis of $V^*$ which is dual to $x_1, \dots, x_d$.  In characteristic $0$, the ring $S^*$ is isomorphic to the polynomial ring $T:=k[y_1, \dots, y_d]$.  Further, if $\chr k=p$, then $T_i\cong S_i^*$ for all $i<p$.  Since we assume $\chr k\ne 2,3$ and focus on ideals of regularity two, we will abuse notation and identify $S^*$ and $T$ throughout. The reader may refer to~\cite[\S 2]{ems-iar-inverse} and ~\cite[\S A2.4]{eis} for further details.

%%%%%%%%%%%%%%
\subsection{Homogeneous ideals of regularity two}\label{subsec:regtwo}
%%%%%%%%%%%%%%
We often consider ideals $I\subseteq S$ which are homogeneous and which have Hilbert function $(1,d,e)$.

\begin{defn}
If $I \subseteq S$ is homogeneous and has Hilbert function $(1,d,e)$, then we refer to $I$ as a \defi{$(1,d,e)$-ideal}.  Note that every zero-dimensional homogeneous ideal of regularity two is a $(1,d,e)$-ideal where $d$ is the embedding dimension of $I$, and where $e=deg(I)-d-1$.
\end{defn}
\noindent The $(1,d,e)$-ideals are parametrized by $\Gr(e,S_2^*)$ in the following way.  Given a $(1,d,e)$-ideal, observe that $I_2^\perp \in \Gr(e,S_2^*)$.  Conversely, given $V\in \Gr(e,S_2^*)$, the ideal $\langle V^\perp \rangle +\mathfrak m^3$ defines a unique $(1,d,e)$-ideal.  By abuse of notation, we will generally consider $\Gr(e,S_2^*)$ to be the subscheme of $H^d_{1+d+e}$ which parametrizes $(1,d,e)$-ideals of $S$.  If $I$ is a $(1,d,e)$-ideal, we often write $I\in \Gr(e,S_2^*)$ in place of $I_2^\perp \in \Gr(e,S_2^*)$.

%%%%%%%%%%%
%%%%%%%%%%%
\section{Smoothable ideals with Hilbert function $(1,d,e)$}\label{sec:smoothable1de}
%%%%%%%%%%%
%%%%%%%%%%%
In this section we describe the locus of smoothable $(1,d,e)$-ideals in two steps.  First, we show that this locus is irreducible, and that it is dominated by a rational map $\pi$ from the smoothable component of the Hilbert scheme of points (Proposition~\ref{prop:ratlmap}).  Second, we give a more concrete description of the image of $\pi$ (Proposition~\ref{prop:smoothparam}). This description will be used in \S\ref{sec:Kvectors} to compute the $\kappa$-vectors of smoothable $(1,d,e)$-ideals.

Recall that, given any ideal $J\subseteq S$ and any weight vector $\mathbf{w}=(w_1, \dots, w_n)\in \mathbb R^{n}_{\geq 0}$, we may define the initial ideal of $J$ with respect to $\mathbf{w}$, which we denote $\initial_{\mathbf{w}}(J)$ (see \cite[\S7.4]{cca} for details).  If $\mathbf{w}=(1,\dots,1)$, then the ideal $\initial_{(1,\dots,1)}(J)$ is homogeneous ideal the standard grading.

Let $U_e$ be the union of the monomial patches $U_\lambda$ such that $U_\lambda\cap \Gr(e,S_2^*)\ne \emptyset$. We claim that the function $J\mapsto \initial_{(1,\dots,1)}(J)$ is regular in $U_e$, thus defines a rational map $\pi: R^d_{1+d+e}\dashrightarrow \Gr(e,S_2^*)$.  If $J\in U_{\lambda}$ and $\lambda$ has Hilbert function $(1,d,e)$ then for every $m\not\in \lambda$ we have $m-\sum_{m'\in \lambda}c^{m}_{m'}m'\in J$. Therefore $\initial_{(1,\dots, 1)}(J)$ contains an ideal $I$ generated by all cubic monomials and $\binom{d+1}{2}-e$ linearly independent quadrics. Thus $I=\initial_{(1,\dots, 1)}(J)$ since any such $I$ has colength $1+d+e$. It follows that $\initial_{(1,\dots, 1)}(J)\in \Gr(e,S_2^*)$.  On $U_e$, $\pi$ is locally a projection away from those $c_{m}^{m'}$ with ${\rm deg}(m')\neq {\rm deg}(m)$ and thus regular.

\begin{prop}\label{prop:ratlmap}
The locus of smoothable $(1,d,e)$-ideals is the image of $\pi$.
\end{prop}

\begin{proof}[Proof of Proposition \ref{prop:ratlmap}]
Observe that $R^d_{1+d+e}\cap \Gr(e,S_2^*)$ belongs to $U_e$ and that $\pi$ is the identity on $R^d_{1+d+e}\cap \Gr(e,S_2^*)$.  Hence we have
\[
R^d_{1+d+e}\cap \Gr(e,S_2^*) \subseteq \pi(U_e\cap R^d_{1+d+e})\subseteq R^d_{1+d+e}\cap \Gr(e,S_2^*),
\]
and we conclude that the image of $\pi$ equals the locus of smoothable $(1,d,e)$-ideals.
\end{proof}

The following proposition provides a more concrete description of the locus of smoothable $(1,d,e)$-ideals.

\begin{prop} \label{prop:smoothparam}
Let $J\subseteq S$ be a generic smoothable ideal of colength $1+d+e$, where $e\leq \binom{d+1}{2}$.  Let $I=\initial_{(1,\dots, 1)}(J)$.  Then:
\begin{enumerate}
\item\label{prop:smoothparam:2}{ $I$ has Hilbert function $(1,d,e)$. Thus it is completely determined by the $e$-dimensional vector space $I_2^{\perp}$.}
\item\label{prop:smoothparam:2}{Up to the action of $GL(d)$, there exist $a_i^{(j)}\in k$ with $1\leq i\leq d$ and $1\leq j\leq e$ such that
\[I_2^{\perp}=\langle q_1,\dots, q_e\rangle\]
with $q_j=\sum_{i=1}^d a^{(j)}_i y_i^2-\left(\sum_{i=1}^d a^{(j)}_i y_i\right)^2$.}
\end{enumerate}
\end{prop}

Let $J$ be a generic ideal of $1+d+e$ reduced points in $\mathbb A^d$, where $e\leq \binom{d+1}{2}$. Acting with translations and with $\GL(d)$ we may assume that $V(J)$ contains the origin $p_0$ and the $d$ canonical basis vectors $p_1, \dots, p_d$ in $\mathbb A^d$.  Moreover $V(J)$ contains $e$ additional points $p_{d+j}=(a_1^{(j)},\dots, a_d^{(j)})$, $1\leq j\leq e$. Let $\widetilde{J}$ be the homogenization of $J$ in $S[x_0]$.  Note that $\widetilde{J}$ is the homogeneous ideal defining $\cup_{i=0}^{d+e} [1:p_i] \subseteq \mathbb P^d$.

\begin{lemma}\label{lem: inv1} With notation as above, we have
\begin{enumerate}
\item{For $1\leq i\leq d$, $1\leq j\leq e$, and any $s\in \mathbb N$, the inverse system $\widetilde{J}^{\perp}$ contains  \[y_0^s,(y_0+y_i)^s, (y_0+a_1^{(j)}y_1+\dots a_d^{(j)}y_d)^s.\]}
\item{$\widetilde{J}_2^{\perp}\cap k[y_1,\dots, y_n]$ contains $q_1,\dots, q_e$ where
\[q_j=\sum_{i=1}^d a^{(j)}_i y_i^2-\left(\sum_{i=1}^d a^{(j)}_i y_i\right)^2.\]
Moreover the $q_j$ are linearly independent.
}
\end{enumerate}
\end{lemma}
\begin{proof} For a point $r=[r_0:\dots: r_d]$ in $\mathbb{P}^d$ with homogeneous ideal $W$, it is immediate that $W_j^{\perp}={\rm span}_k((r_0y_0+\dots+r_dy_d)^j)$. If $W^{(1)},\dots, W^{(1+d+e)}$ are the ideals of the points of $V(\widetilde{J})$, then we have
\[\widetilde{J}_j^{\perp}=\left( \bigcap_{i=1}^{1+d+e} W^{(i)}\right)_j^{\perp} \supseteq (W^{(1)})_j^{\perp}+\dots+(W^{(1+d+e)})_j^{\perp}. \]
Hence $(1)$. For $(2)$, let $l_j:=y_0+a_1^{(j)}y_1+\dots a_d^{(j)}y_d$ and observe that
\[q_j=y_0^2-l_j^2+\sum_{i=1}^da_i^{(j)}\left((y_0+y_i)^2-y_0^2\right). \]
Thus, the $q_j$ belong to $\widetilde{J}_2^\perp \cap k[y_1, \dots, y_d]$.  To prove that the $q_j$ are linearly independent it suffices to show that all squares of linear forms in part $(1)$ are linearly independent.  Note that the squares of linear forms are precisely the points in the image of the second Veronese embedding of $\mathbb{P}^d$ via the complete linear system $|2H|$. This image does not lie in any hyperplane of $\mathbb{P}^{\binom{d+2}{2}-1}$, and therefore the squares of a generic set of $m\leq \binom{d+2}{2}$ linear forms in $d$ variables are linearly independent.  Since $e\leq \binom{d+1}{2}$, it then follows that the $q_j$ are linearly independent.
\end{proof}

\begin{proof}[Proof of Proposition \ref{prop:smoothparam}] Part $(1)$ follows from Lemma \ref{lem: inv1}.  For part $(2)$, note that for any translation $T$, any $G\in GL(d)$, and any polynomial $g\in S$, we have $${\rm in}_{(1,\dots, 1)}(G(T(g))=G({\rm in}_{(1,\dots, 1)}(g)).$$  Thus we may assume that $V(J)$ contains the origin $p_0$ and the $d$ canonical basis vectors $p_1, \dots, p_d$. Moreover, $V(J)$ contains $e$ additional points $p_{d+j}$ whose coordinates we label $(a_1^{(j)},\dots, a_d^{(j)})$, $1\leq j\leq e$.

Now let $h_2\in {\rm in}_{(1,\dots, 1)}(J)_2$ and let $g=h_2+h_1+h_0\in J$ with $\deg(h_i)=i$. Define $\tilde{g}=h+x_0h_1+x_0^2h_2\in \widetilde{J}$.  For any $\phi\in \widetilde{J}_2^{\perp}\cap k[y_1,\dots, y_n]$ we have
$\phi h_2=\phi (\tilde{g})=0$. Therefore
\[\widetilde{J}_2^{\perp}\cap k[y_1,\dots, y_n]\subseteq {\rm in}_{(1,\dots, 1)}(J)_2^{\perp}\]
Applying Lemma~\ref{lem: inv1}, we see that $\langle q_1,\dots, q_e\rangle \subseteq {\rm in}_{(1,\dots, 1)}(J)_2^{\perp}$. Part $(2)$ of the proposition follows since the $q_j$ are linearly independent and the right hand side has dimension $e$.
\end{proof}

Using Proposition~\ref{prop:ratlmap}, we now estimate the dimension of the locus of smoothable ideals in $\Gr(3,S_2^*)$.  This is an important ingredient in the proof of Theorem~\ref{thm:genwaldo}. We briefly review a coordinate system for $R^d_n$ introduced in~\cite[(2.21)]{haiman-catalan}.  Suppose $M$ is a monomial ideal of colength $n$ with standard monomials $\lambda$, and suppose that $J$ is an ideal such that $V(J)$ consists of $n$ distinct points $a^{(1)},\dots, a^{(n)}$ with coordinates $a_i^{(j)}$ for $1\leq i\leq d$. Fix an order $(m_1,\dots, m_n)$ on the set $\lambda$ and define $\Delta_{\lambda}=\det\left([m_i(a^{(j)})]_{i,j}\right)$. If $J\in U_{\lambda}$, then we can express the coordinates $c_{m'}^m$ in terms of the $a_i^{(j)}$ using Cramer's rule as:
\[c_{m'}^m=\frac{\Delta_{\lambda-m'+m}}{\Delta_{\lambda}}\]
where $\lambda-m'+m$ is the ordered set of monomials obtained from $\lambda$ by
replacing $m'$ with $m$. Glueing over the various $U_\lambda$, these quotients
determine a rational map $\Delta: (\mathbb{A}^d)^n \dasharrow R^d_n$ which is regular when the points $a^{(j)}$ are all distinct.

\begin{lemma}\label{lem:M2}
If $4\leq d\leq 11$, then $R^d_{d+4}\cap \Gr(3,S_2^*)$ has codimension at most ${d-2}\choose{2}$ in $\Gr(3,S_2^*)$ and codimension at most $12$ in $R^{d}_{d+4}$.
\end{lemma}
\begin{proof} We have a rational map  $g:=\pi\circ\Delta: (\mathbb A^d)^n \dashrightarrow \Gr(e,S_2^*)$ where $\pi$ is the map introduced in Proposition~\ref{prop:ratlmap}. Let $Y$ be the domain of definition of $g$.  If $q\in Y$, then the dimension of every component of the fiber $Y_{g(q)}$ is at least $\dim(Y)-\dim(g(Y))$.  It follows that for any $q\in Y$ we have
\[\dim(g(Y))\geq \dim(Y)-\dim(T_{q}Y_{g(q)}).\]
This inequality allows the computation of explicit lower bounds for the dimension of the locus of smoothable $(1,d,3)$-ideals for small values of $d$.  Computing $\dim(T_qY_{g(q)})$ in Macaulay2~\cite{M2} with $k=\mathbb Q$ yields the following table:
\[
\begin{array}{|c|c|c|c|c|}
\hline
d & n & \leq \dim(g(Y)) & \dim(\Gr(3,S_2^*)) & \binom{d-2}{2} \\
\hline
4 & 8 & 20 & 21 & 1\\ \hline
5 & 9 & 33 & 36 & 3 \\ \hline
6 & 10 & 48 & 54 & 6\\ \hline
7 & 11 & 65 & 75 & 10\\ \hline
8 & 12 & 84 & 99 & 15\\ \hline
9 & 13 & 105 & 126 & 21\\ \hline
10 & 14 & 128 & 156 & 28\\ \hline
11 & 15 & 153 & 189 & 36\\
\hline
\end{array}
\]
Thus, for $4\leq d\leq 11$ the codimension of the intersection $R^d_{d+4}\cap \Gr(3,S_2^*)$ in $\Gr(3,S_2^*)$ is at most ${d-2\choose 2}$.  By semicontinuity of fiber dimensions, the lower bound obtained by computation over $\mathbb Q$ holds over a field of any characteristic. Finally, the last statement of the proposition follows since the dimension of $R^{d}_{d+4}$ is $d(d+4)$.
\end{proof}

%%%%%%%
%%%%%%%
\section{$\kappa$-vectors and Betti numbers }\label{sec:Kvectors}
%%%%%%%
%%%%%%%

In this section, we elaborate on the definition of $\kappa$-vector, and we discuss some of its elementary properties. We compute $\kappa_0$ and $\kappa_1$ of a generic $(1,d,e)$-ideal and of a generic {\em smoothable} $(1,d,e)$-ideal.  These computations will be used in the proofs of Theorems~\ref{thm:1denonsmoothable}, \ref{thm:minexamples} and~\ref{thm:genwaldo}.  We also reinterpret the entries of the $\kappa$-vector of $I$ as the graded Betti numbers of a certain module constructed from $I$.  This interpretation reveals surprising dependencies among the entries of the $\kappa$-vector.

Since ${\rm char}(k)\neq 2$, we will think of elements of $S_2^*$ as symmetric linear transformations from $S_1$ to $S_1^*$.  Let $I\subseteq S$ be a homogeneous ideal with Hilbert function $(1,d,e)$ and choose a basis $\mathbf{A}=(A_1,\dots, A_e)$ of $I_2^{\perp}$ .

\begin{defn}\label{defn:kappa-vec}
For $0\leq j \leq e-1$ let $\psi_j(\mathbf A)$ be the linear map from $S_1\otimes \wedge^j I_2^\perp$ to $S_1^*\otimes \wedge^{j+1} I_2^\perp$ given by:
 $$\psi_j(\mathbf A)(u \otimes E)=u\otimes (E\wedge \mathbf A):=\sum_{i=1}^e A_i(u)\otimes (E\wedge A_i).$$
 We define the $\kappa$-vector $\kappa(I)=(\kappa_0(I), \dots, \kappa_{e-1}(I))$ by
 $$\kappa_j(I):=\rank \left(\psi_j( \mathbf A)\right).$$
\end{defn}
\noindent Note that $A_i$ is playing different roles on the two sides of the tensor.  On the left-hand side, $A_i\in \Hom(S_1,S_1^*)$, so that $A_i(u)\in S_1^*$.  On the right-hand side, $A_i$ is an element of the vector space $I_2^\perp$, so that $E\wedge A_i\in \wedge^{j+1} I_2^\perp$.

\begin{lemma}\label{lem: inv} The $\kappa$-vector has the following properties:
\begin{enumerate}
\item{The $\kappa$-vector $\kappa(I)$ does not depend on the choice of basis of $I_2^\perp$ and is invariant under linear changes of coordinates on $S$.}
\item Each $\kappa_j$ is lower semicontinuous on $\Gr(e,S_2^*)$.
\item {The $\kappa$-vector is symmetric: $\kappa_j=\kappa_{e-j-1}$ for every $j\leq \lfloor{\frac{e}{2}}\rfloor$.}
\item {Let $e=3 \mod 4$ and let $e=4f+3$.  Assume further that $\binom{4f+3}{2f+1}$ and $d$ are odd.  Then $\kappa_{2f+1}(I)< d\binom{4f+3}{2f+1}$.}
\end{enumerate}
\end{lemma}
\begin{proof}
(1) Suppose that $\alpha \in GL(e)$ is the change of basis from $\mathbf A$ to some other basis $\alpha(\mathbf{A})$.  Set $\Gamma_j:=  Id_{d} \otimes\left(\wedge^{j} \alpha \right)$.  Note that:
$$\Gamma_j^{-1} \psi_j(\alpha(\mathbf A)) \Gamma_{j+1}=\psi_j(\mathbf A)$$
It follows immediately that $\kappa_j(I)$ does not depend on our choice of basis of $I_2^\perp$.

Next let $\beta\in GL(d)$ and let $B_j:= \beta \otimes \wedge^j Id_{e}$.  Then we have:
$$B_j^{t} \psi_i(\mathbf A) B_{j+1}=\psi_i(\beta^t \mathbf A \beta)$$
Thus $\kappa_i(I)$ is invariant under the $GL(d)$-action.

(2) For a fixed sequence $\vec s \in \mathbb N^e$, the locus $\{ I \in \Gr(e,S_2^*) | \kappa(I)\leq \vec s\}$ is cut out by determinantal conditions on the maps $\psi_j(I)$ used in the definition of $\kappa(I)$.  The ranks of these linear maps can not increase under specialization of the vector space $I_2^\perp$, thus yielding the desired semicontinuity.

(3) The matrices $A_i$ are symmetric, and thus $\psi_{j}(\mathbf{A})^t=\pm \psi_{e-1-j}(\mathbf{A})$.

(4) The conditions guarantee that $\psi_{2f+1}(\mathbf{A})$ is a skew-symmetric matrix of odd size, and hence it cannot have full rank.
\end{proof}

We now compute $\kappa_0$ and $\kappa_1$ for some $(1,d,e)$-ideals.

\begin{prop}\label{prop:genKvec}\label{prop:smoothKvec}
Let $e\geq 3$, let $I$ be a generic $(1,d,e)$-ideal, and let $I'$ be a generic smoothable $(1,d,e)$-ideal.
\begin{enumerate}
    \item\label{prop:genKvec:1} ({\bf Generic case})
        \begin{itemize}
            \item $\kappa_0(I)=d.$
            \item  $\kappa_1(I)=ed$ unless $e=3$ and $d$ is odd, in which case $\kappa_1(I)=3d-1$.
        \end{itemize}
    \item\label{prop:genKvec:2}({\bf Generic smoothable case})
        \begin{itemize}
            \item  $\kappa_0(I')=d$.
            \item  If $d\geq \binom{e}{2}$, then $\kappa_1(I')\leq (e-1)d+\binom{e}{2}$.  Further, if $e=3$ then $\kappa_1(I')=2d+2$.
            \item  $\kappa_i(I')\leq (d+e)\binom{e-1}{i}$ for $i=1, \dots, e-1$.
        \end{itemize}
\end{enumerate}
\end{prop}

\begin{proof}
Throughout this proof we will use the isomorphism $S_1\rightarrow S_1^*$ given by $x_i\mapsto y_i$ so that the compositions $\psi_{j+1}\circ\psi_{j}$ are well defined. This allows us to define a sequence of vector spaces $\mathbb{K}(\mathbf{A}):=(S_1^{*}\otimes\bigwedge^{\bullet}(I_2)^{\perp},\psi_{\bullet})$.

\eqref{prop:genKvec:1}
Since $I$ is generic, we may assume that $I_2^\perp$ contains a quadric $A_e$ of full rank.  Therefore $\kappa_0(I)=d$. Moreover, $\psi_j(\mathbf{A})$ has the block form
\[
\psi_j^s =
\left(
\begin{array}{cc}
-\psi_{j-1}^{s-1} & A_s\otimes Id_e\\
0 & \psi_j^{s-1}\\
\end{array}
\right)
\]
where $\psi_j^{\ell}$ is $\psi_j(A_1,\dots, A_{\ell})$. Since the $\kappa$-vector is independent of the coordinates chosen in $S$ we may assume that $A_e$ equals the identity matrix $I$ so that
 \[
\left(
\begin{array}{cc}
-\psi_{j-1}^{e-1} & A_e\otimes Id_e\\
0 & \psi_j^{e-1}\\
\end{array}
\right)
\left(
\begin{array}{cc}
Id_e & 0\\
\psi_{j-1}^{e-1} & Id_e\\
\end{array}
\right)
=
\left(
\begin{array}{cc}
0 & Id_e\\
\psi_{j}^{e-1}\psi_{j-1}^{e-1} & \psi_{j}^{e-1}\\
\end{array}
\right)
\]
and thus for every $j\geq 1$,
\[
\kappa_j(I)=d{{e-1}\choose{j}}+\rank\left( \psi_{j}^{e-1}\circ\psi_{j-1}^{e-1}\right).
\]
For $j=1$ we distinguish two cases. If $e=3$, then $\psi_1^{e-1}\psi_0^{e-1}$ coincides with the commutator $[A_2,A_1]$. If we choose $A_2$ to be generic antidiagonal and $A_1$ to be generic diagonal, then $[A_2,A_1]$ has rank $d$ (resp. $d-1$) if $d$ is even (resp. odd). On the other hand, if $e\geq 4$, then the composition $\psi_1^{e-1}\psi_0^{e-1}$ is a $d\times {e\choose{2}}d$ block matrix containing  $[A_2,A_1]A_1\wedge A_2+[A_3,A_2]A_2\wedge A_3$.  Choosing $A_2$ and $A_1$ as in the case $e=3$, we see that this matrix has full rank $d$ when $d$ is even, or rank at least $d-1$, when $d$ is odd.  In the odd case, all entries of $[A_2,A_1]$ in the middle row are zero.  For generic $A_3$, the commutator $[A_3,A_2]$ will have nonzero entries in the middle row.  Hence,  $[A_2,A_1]A_1\wedge A_2+[A_3,A_2]A_2\wedge A_3$ has full rank $d$.

\eqref{prop:genKvec:2}  From Proposition~\ref{prop:smoothparam} part \eqref{prop:smoothparam:2}, we see that $(I'_2)^\perp$ contains a quadric of full rank.  Hence $\kappa_0(I')=d$.  Now, let $d\geq \binom{e}{2}$.  Since $I'$ is generic smoothable, Proposition~\ref{prop:ratlmap} implies that we may choose an ideal $J$ of distinct points such that $I'=\initial_{(1,\dots,1)}(J)$.  We will show that
\[\kappa_1(I')\leq (e-1)d+\binom{e}{2}.\]
By symmetry of the $\kappa$-vector, it suffices to show that the above inequality holds for $\kappa_{e-2}$. Lemma~\ref{prop:smoothparam} implies that, after possibly changing coordinates on $S_1$, the subspace $(I')_2^{\perp}$ has a basis $A_1,\dots, A_e$ consisting of matrices $A_i=E_i-D_i$ where $D_i$ is the diagonal matrix with entries $(D_i)_{kk}=a^{(i)}_k$ and $E_i$ is the rank one matrix $\vec{a}_{(i)}{\vec{a}_{(i)}}^{t},$ where $\vec{a}_{(i)}=(a^{(i)}_1,\dots a^{(i)}_d)$. Moreover, $\psi_{e-2}(\mathbf{A})=\psi_{e-2}(\mathbf{E})-\psi_{e-2}(\mathbf{D})$. As a result
\begin{equation}\label{eq:gensmoothable}
\kappa_{e-2}(\mathbf{A})\leq\dim \left({\rm Im}(\psi_{e-2}(\mathbf{D}))+{\rm Im}(\psi_{e-2}(\mathbf{E})) \right)
\end{equation}
\[
\qquad =\kappa_{e-2}(\mathbf{D})+\kappa_{e-2}(\mathbf{E})-{\rm dim (W)}
\]
where $W={\rm Im}(\psi_{e-2}(\mathbf{D}))\cap{\rm Im}(\psi_{e-2}(\mathbf{E}))$. To prove the theorem we will estimate the terms appearing in the right hand side.

First note that $\psi_{e-2}(\mathbf{E})$ is a block matrix of the form
\begin{equation*}
\bordermatrix{
& \widehat{1,2}  & \widehat{1,3} & \dots & \widehat{e-1,e} \cr
\widehat{1} & \pm E_2 & \pm E_3 & \dots & \vdots \cr
\widehat{2} & \pm E_1 & 0 & \dots  & \vdots \cr
\widehat{3} & 0 &  \pm E_1& \dots  & \vdots \cr
\vdots & \vdots & \vdots & \dots  & \vdots \cr
\widehat{e} & 0 & 0 & \dots & \pm E_{e-1}
}
\end{equation*}
having ${{e}\choose{2}}$ block columns of rank at most two. Hence $\kappa_{e-2}(\mathbf{E})\leq \min\{de,2{{e}\choose{2}}\}$.
On the other hand the $D_i$ are diagonal matrices and thus $\mathbb{K}(\mathbf{D})$ is isomorphic to the direct sum of $e$ copies of the reduced cohomology chain complex of the $d$-simplex. It follows that $\mathbb{K}(\mathbf{D})$ is exact and moreover, since the $a_i$ are generic, that $\kappa_{e-2}(\mathbf{D})=d(e-1)$.
Now, let $\eta$ be the matrix obtained from $\psi_{e-2}(\mathbf{E})$ by extracting the first $2$ columns from each block of $\psi_{e-2}(\mathbf{E})$. Note that $\eta$ is injective and that ${\rm Im}(\eta)={\rm Im}(\psi_{e-2}(\mathbf{E}))$. By exactness of $\mathbb{K}(\mathbf{D})$, $W$ is isomorphic to the kernel of the composition $\psi_{e-1}(\mathbf{D})\circ \eta$.
This composition is a $d\times 2\binom{e}{2}$ matrix consisting of $\binom{e}{2}$ blocks each of which is a $d\times 2$ matrix of the form
%This composition is a (block) row matrix with ${{e}\choose{2}}$ blocks with $d$ rows and $2$ columns each of the form

\[
\left(D_i\vec{a_{(j)}}{{a_1}^{(j)}}-D_j\vec{a_{(i)} } {a_1}^{(i)},D_i\vec{a_{(j)}}{{a_2}^{(j)}}-D_j\vec{a_{(i)} } {a_2}^{(i)}\right).
\]
Its range lies in the span of the $d\times 2\binom{e}{2}$ matrix of the form $\left(D_i\vec{a_{(j)}}, D_j\vec{a_{(i)}}\right)$. The latter matrix has rank $\min(d,{{e}\choose{2}})$ since $D_i\vec{a_{(j)}}=D_j\vec{a_{(i)}}$. As a result we have
\[
\kappa_{e-2}(\mathbf{E})-\dim_k W\leq \min\{de,2\binom{e}{2}\}-\left(2\binom{e}{2}-\min\{ d,\binom{e}{2} \} \right).
\]
Since we assume that $d\geq \binom{e}{2}$, this simplifies to
\[
\kappa_{e-2}(\mathbf{E})-\dim_k W\leq 2\binom{e}{2}-(2\binom{e}{2}-\binom{e}{2})=\binom{e}{2}
\]
Combining this inequality with \eqref{eq:gensmoothable}, we obtain the upper bound from the proposition.

Now we consider $\kappa_1$ in the case $e=3$. Note that the upper bound given is $2d+3$, but since $\psi_1(\mathbf{A})$ is skew-symmetric, this implies that $\kappa_1(I')\leq 2d+2$.  To verify the desired equality, we produce an example.  Using notation as in Lemma~\ref{lem: inv1}, we specialize to the case $p_{d+1}=(1,\dots,1), p_{d+2}=(1,1,0,\dots,0)$ and $p_{d+3}=(0,\dots,0,1,1)$.  We claim that the first $2d+1$ rows of the corresponding matrix $\psi_1(\mathbf{A})$ are linearly independent.  Since $A_1$ has rank $d$, the first $2d$ rows are linearly independent.  Let $w$ be the vector:
\[
w:=(-d+2)x_d\otimes A_1+(-d+4)x_1\otimes A_3+(d-1)x_1\otimes A_3+\sum_{i=1}^{d-1}x_i\otimes A_1
\]
The vector $w$ belongs to the kernel of the submatrix spanned by the first $2d$ rows of $\psi_1(\mathbf{A})$, but not to the kernel of the first $2d+1$ rows.  Thus $\psi_1(\mathbf{A})$ has rank at least $2d+1$; since it is skew-symmetric, it therefore has rank at least $2d+2$.  This completes the proof for $\kappa_1$.

Finally, we consider the general case of $\kappa_i(I')$.  We think of $(A_1, \dots, A_e)$ as an element of $k^e\otimes k^d\otimes k^d$ via the injection $\Sym_2(k^d)\subseteq k^d\otimes k^d$.  It is clear that Definition~\ref{defn:kappa-vec} could be extended to any $3$-tensor in $k^e\otimes k^d\otimes k^d$.  Further, observe that if $x,x'\in k^e\otimes k^d\otimes k^d$ then $\kappa_i(x+x')\leq \kappa_i(x)+\kappa_i(x')$. Proposition~\ref{prop:smoothparam} part \eqref{prop:smoothparam:2} implies that $\mathbf{A}$ can be written as the sum of $d+e$ pure $3$-tensors in $k^e\otimes k^d\otimes k^d$.
Hence, to prove the inequality for $\kappa_i(I')$, it suffices to compute $\kappa_i(x)$ in the case that $x$ is a pure tensor.  We may express any pure $3$-tensor as the sequence $(A_1, 0, \dots, 0)$ where $A_1$ is a symmetric rank $1$ matrix.  It follows that $\kappa_i(x)=\binom{e-1}{i}$, which proves the claim for $\kappa_i(I')$.

\end{proof}

\begin{remark}
A similar argument as in the proof of Proposition~\ref{prop:genKvec} part \eqref{prop:genKvec:1} shows that, if $e\geq 3$ and $j\leq \frac{e}{2}$, then
\[\kappa_j(I)\geq d\left({{e-1}\choose{j}}+{{e-2}\choose{j-1}}\right)\]
We omit the proof since it will not be used in this paper, and because computer experiments indicate that $\kappa_j$ is considerably larger.
\end{remark}

Now we reinterpret the $\kappa$-vector in terms of the graded Betti numbers of a certain module. Given $I\in \Gr(e,S_2^*)$ with basis $\mathbf{A}=(A_1,\dots, A_e)$, consider the $R:=k[z_1,\dots, z_e]$-graded module $M(\mathbf{A})$ whose graded pieces are  $M_0=S_1$ , $M_1=S_1^*$ and $M_j=0\text{ if } j\not\in \{0,1\}$.  The action $z_i\cdot (u_0+u_1)$ is defined by $A_i(u_0)$. The relationship between the Betti numbers of $M$ and the $\kappa$-vector is summarized in the following proposition. To simplify the formulas we set $\kappa_{-1}=\kappa_e=0$.

\begin{prop}\label{prop:bettiKvec}
For $0\leq i\leq e$ the graded Betti numbers of $M(\mathbf{A})$ satisfy
\[b_{i,s}=\begin{cases}
d{{e}\choose{i}}-\kappa_{i-1}(I)\text{, if $s=i$}\\
d{{e}\choose{i}}-\kappa_{i}(I)\text{, if $s=i+1$}\\
0\text{, else}
\end{cases}
\]
Conversely, the components of the $\kappa$-vector of $I$ can be expressed in terms of the Betti numbers of $M$ as
\[\kappa_j(I)=d{{e}\choose{j}}-b_{j,j+1}(M).\]
\end{prop}

\begin{proof} Recall that $b_{i,s}(M)={\rm dim}\Tor^i(M,k)_s$~\cite[Prop.\ 1.7]{eis-syzygy}. The right hand side is the $s$-graded piece of the $i$-th homology of the complex $\mathbb{F}:=\mathbb{K}(z_1,\dots, z_e)\otimes_R M$ obtained by tensoring the Koszul complex on $z_1,\dots, z_e$ with the $T$-module $M$. In our case the complex $\mathbb{F}$ is
\[ \dots \rightarrow \bigwedge^{e-i-1}I_2^{\perp}\otimes_k M(-i-1)\rightarrow \bigwedge^{e-i}I_2^{\perp}\otimes_k M(-i) \rightarrow  \bigwedge^{e-i+1}I_2^{\perp}\otimes_k M(-i+1)\rightarrow \dots \]
and in particular the graded component of $\mathbb{F}$ in degree $i$ is the complex:
\[
\text{degree } i: \quad \quad \dots \rightarrow 0\rightarrow \bigwedge^{e-i}I_2^{\perp}\otimes_k S_1\rightarrow  \bigwedge^{e-i+1}I_2^{\perp}\otimes_k S_1^*\rightarrow 0\rightarrow \dots
\]
Similarly, the graded component of $\mathbb{F}$ in degree $i+1$ is:
\[
\text{degree } i+1: \quad \dots \rightarrow 0\rightarrow \bigwedge^{e-i}I_2^{\perp}\otimes_k S_1\rightarrow  \bigwedge^{e-i+1}I_2^{\perp}\otimes_k S_1^*\rightarrow 0\rightarrow \dots
\]
%\[ \bigwedge^{e-i}I_2^{\perp}\otimes_k S_1 \rightarrow \bigwedge^{e-i}I_2^{\perp}\otimes_k S_1^*\rightarrow 0.\]
The differentials of these complexes are $\psi_{e-i}(\mathbf{A})$ and  $\psi_{e-i-1}(\mathbf{A})$ respectively.  The formulas in the proposition then follow from the symmetry of the $\kappa$-vector.
\end{proof}

Using the notation from \S\ref{subsec:Betti}, Proposition~\ref{prop:bettiKvec} may be summarized by writing $\beta(M(\mathbf{A}))$ as
\smallskip
\[
\begin{pmatrix}
d\binom{e}{0} & d\binom{e}{1}-\kappa_0(I) & d\binom{e}{2}-\kappa_1(I) & \dots & d\binom{e}{e}-\kappa_{e-1}(I) \vspace{.2cm} \\
d\binom{e}{0}-\kappa_0(I) & d\binom{e}{1}-\kappa_1(I) & d\binom{e}{2}-\kappa_2(I) & \dots & d\binom{e}{e}
\end{pmatrix}.
\]

%%%%%%%%%%%%%%%%%%%%%%%%%%
%%%%%%%%%%%%%%%%%%%%%%%%%%
\subsection{Boij-S\"oderberg theory and $\kappa$-vectors}\label{subsec:boijsod}
%%%%%%%%%%%%%%%%%%%%%%%%%%
%%%%%%%%%%%%%%%%%%%%%%%%%%
Using Boij-S\"{o}derberg theory, we show that the entries of the $\kappa$-vector are interdependent.   First we review the relevant facts from Boij-S\"{o}derberg theory.  Given a graded $S$-module $M$, let $\mathbb{F}$ be the graded minimal free resolution of $M$. The \defi{graded Betti numbers} of $M$ are the integers $b_{i,j}$ defined by $\mathbb{F}_i=\bigoplus_j S(-j)^{b_{i,j}}$. The \defi{Betti diagram} of $M$, denoted $\beta(M)$, is the matrix
\[
\beta(M):=
\begin{pmatrix}
b_{0,0}& b_{1,1} & \dots & b_{p,p}\\
b_{0,1}& b_{1,2} & \dots & b_{p,p+1}\\
\vdots &\vdots &  & \vdots \\
b_{0,r}& b_{1,1+r} & \dots & b_{p,p+r}
\end{pmatrix}
\]
Boij-S\"oderberg theory provides an algorithm for expressing the Betti diagram of a module as a positive rational combination of simple building blocks called \defi{pure diagrams}.  See the introduction of ~\cite{eis-schrey} for an overview.

Let $m=\min \{ i\geq 0   | \kappa_i<d\binom{e}{i} \}$.  Let $M=M(\mathbf{A})$ be the graded module associated to some basis $\mathbf{A}$ of $I_2^\perp$, as defined in Proposition \ref{prop:bettiKvec}.  The Betti diagram of $M$ then has the following shape.
\[
\bordermatrix{
		& 0 & 1 & \dots & m-1 & m & \dots & e-1-m & e-m& \dots & e-1 \cr
		 & *&*&\dots &*&*&\dots&*&-&\dots &-\cr
		& -&-&\dots &-&*&\dots&*&*&\dots &*
		}
\]
A $*$ represents a nonzero entry, and a $-$ represents a zero.  Observe that $m$ is the smallest integer for which $b_{m,m+1}\neq 0$.

\begin{prop}\label{prop:boijsod}
With notation as above, the graded Betti numbers of $M$ satisfy the following inequalities:
\[
b_{i,i+1}(M)\geq b_{m,m+1}(M)(i+1-m) \frac{(m+1)!(e-m)!}{(i+1)!(e-i)!}
\]
for all $m\leq i \leq \lfloor \frac{e-1}{2} \rfloor$.
Equivalently, the $\kappa$-vector of $I$ satisifes the inequalities:
\[
\kappa_i(I) \leq d\binom{e}{i} - \left(d\binom{e}{m}-\kappa_m(I) \right) (i+1-m) \frac{(m+1)!(e-m)!}{(i+1)!(e-i)!}
\]
for all $m\leq i \leq \lfloor \frac{e-1}{2} \rfloor$.
\end{prop}
\begin{proof}
This proof uses the terminology from the introduction of~\cite{eis-schrey}.  Let $\delta_m$ be the degree sequence $(0,1,\dots, m-1,m+1, \dots, e+1)\in \mathbb N^{e+1}$ and let $D$ be the unique pure diagram corresponding to $\delta_m$ with $b_{m,m+1}(D)=b_{m,m+1}(M)$.  From the Herzog-K\"uhl equations~\cite[p.\ 2]{eis-schrey} it follows that $D$ is the diagram:
\[
\frac{b_{m,m+1}(M)}{\binom{e+1}{m+1}} \cdot
\left(\begin{smallmatrix}
m\binom{e+1}{0}& \dots & 1\binom{e+1}{m-1}& 0 & \dots&0 & \dots & 0\\
0&\dots & 0 & 1\binom{e+1}{m+1}  & \dots&(i+1-m)\binom{e+1}{i+1}  & \dots & (e+1-m)\binom{e+1}{e+1}
\end{smallmatrix}\right).
%\begin{pmatrix}
%m\binom{e+1}{0}&(m-1)\binom{e+1}{1} & \dots & 1\binom{e+1}{m-1}& 0 & 0 &\dots & 0\\
%0&0&\dots & 0 & 1\binom{e+1}{m+1} & 2\binom{e+1}{m+2} & \dots & \frac{e+1-m}{m} (e+1-m)\binom{e+1}{e+1}
%\end{pmatrix}
\]
Since the Betti diagram of $M$ is symmetric, the Decomposition Algorithm of \cite{eis-schrey} implies that the difference of diagrams $\beta(M)-D$ will be a new diagram consisting entirely of nonnegative entries.  In particular, for every $i\geq m$ we have that:
\begin{align*}
b_{i,i+1}(M)&\geq b_{i,i+1}(D)\\
&=\frac{b_{m,m+1}(M)}{\binom{e+1}{m+1}} \left( (i+1-m) \binom{e+1}{i+1} \right).
\end{align*}
Simplifying the right-hand side proves the first statement.  The second statement then follows by applying Proposition \ref{prop:bettiKvec}.
\end{proof}

It would be interesting to determine all sequences which equal the $\kappa$-vector of some ideal.  The previous proposition shows that many symmetric vectors in $\mathbb N^{e}$ do not occur as the $\kappa$-vector of some ideal.
\begin{example}
Let $d=5$ and $e=5$, and let $I\in \Gr(5,S_2^*)$.  If $I$ is generic then $\kappa(I)=(5,25,50,25,5)$ and the Betti diagram of $M$ is:
$$
\begin{pmatrix}
5&20&25&-&-&-\\
-&-&-&25&20&5
\end{pmatrix}.
$$
Imagine, however, that we choose $I$ so that $\kappa(I)=(5,22,\kappa_3,22,5)$ for some $\kappa_3\leq 50$.  Then the Betti diagram of $M$ looks like:
$$
\begin{pmatrix}
5&20&28&50-\kappa_3&3&-\\
-&3&50-\kappa_3&28&20&5
\end{pmatrix}.
$$
Boij-S\"oderberg theory implies that the entries of the following difference of diagrams must be nonnegative:
$$
\begin{pmatrix}
5&20&28&50-\kappa_3&3&-\\
-&3&50-\kappa_3&28&20&5
\end{pmatrix}
-
\frac{1}{5}\begin{pmatrix}
1&-&-&-&-&-\\
-&15&40&45&24&4
\end{pmatrix}.
$$
Proposition \ref{prop:boijsod} implies that $50-\kappa_3\geq 8$, or that $\kappa_3$ is at most $42$.
\end{example}
%%%%%%%%%
%%%%%%%%%
\section{$\kappa$-cycles}\label{sec:Kcycles}
%%%%%%%%%
%%%%%%%%%
In this section, we use the $\kappa$-vector to define $\GL(S_1)$-equivariant subsets of the grassmanian $\Gr(e,S_2^*)$. These will be used in the proof of Theorem~\ref{thm:genwaldo}.  Recall that the scheme $\Gr(e,S_2^*)$ is equivariant with respect to the $\GL(S_1)$-action on $\Lambda^e \Sym_2(S_1^*)$.  More explicitly, if $\mathbf{A}\in \Gr(e,S_2^*)$ is an $e$-dimensional vector space, then $g\in \GL(S_1)$ acts by $g\cdot \mathbf{A}\mapsto g\mathbf{A}g^t$.

\begin{defn} Let $\vec{s}=(s_0,\dots, s_{e-1})$ be a sequence of positive integers and let $I\in\Gr(e,S_2^*)$. We say that $\kappa(I)\leq \vec{s}$ if $\kappa_i(I)\leq s_i$ for all $i$. The \defi{$\kappa$-cycle} $\Xi(\vec{s})$ is defined as the closed subset of the Grassmannian ${\rm Gr}(e,S_2^*)$ given by
\[
\Xi(\vec{s})=\{I\in \Gr(e,S_2^*): \kappa(I)\leq \vec{s}\}.
\]
\end{defn}

\begin{example}
Let $e=1$.  Then for any $d'\leq d$, the $\kappa$-cycle $\Xi(d')$ corresponds to the determinantal variety of symmetric $d\times d$-matrices of rank $\leq d'$.
\end{example}

Lemma~\ref{lem: inv} implies that each $\kappa$-cycle is equivariant under the $\GL(S_1)$-action.  These $\kappa$-cycles play an important role in describing the intersections between components of Hilbert schemes of points.  More specifically, Proposition~\ref{prop:genKvec} part~\eqref{prop:genKvec:2} shows that every smoothable $(1,d,e)$-ideal belongs to the $\kappa$-cycle $\Xi(d,2d+2,d)$.  This leads us to investigate the geometry of $\kappa$-cycles of the form $\Xi(d,2d+2,d)$.

\begin{defn}
A vector space of quadrics $V\in \Gr(e,S_2^*)$ is \defi{purely singular} if for every $A\in V$, $\rank(A)<d$.
\end{defn}
Note that $\det(A)$ defines a hypersurface in $\Spec k[a_{ij}]=\mathbb A^{\binom{d+1}{2}}$.  Let $P\subseteq \Gr(3,S_2^*)$ be the locus of purely singular vector spaces.  Then $P\subseteq \Gr(3,S_2^*)$ is the Fano variety of $3$-planes through the origin contained in the hypersurface $V(\det(A))$.
\begin{prop}\label{prop:genwaldoKcycles}
Let $\chr k=0$.  The locus $\Xi(d,2d+2,d)- P$ is an irreducible subset of $\Gr(3,S_2^*)$ of codimension $\binom{d-2}{2}$.
\end{prop}
\begin{proof}
Let
\[
\mathcal T=\mathbb A^{\binom{d+1}{2}}\times \mathbb A^{\binom{d+1}{2}}
 \]
parametrize pairs $(B,C)$ of symmetric $d\times d$ matrices.  Note that $\mathcal T=\Spec (k[b_{ij}, c_{ij}])$ for $1\leq i \leq j \leq d$, where we think of $b_{ij}$ as the entries of $B$ and $c_{ij}$ as the entries of $C$.  We have a surjective rational map
\[
p: \mathcal T\times GL(S_1) \dashrightarrow \Gr(3,S_2^*)-P
\]
which sends
\[
((B,C),g)\mapsto \text{span}\{gIdg^t, gBg^t, gCg^t\}\in \Gr(3,S_2^*)
\]
Let $X\subseteq \mathcal T$ be the determinantal subscheme defined by $\rank(BC-CB)\leq 2$.

We first claim that $X$ is an integral subscheme of codimension $\binom{d-2}{2}$ in $T$.  If $N$ is a skew-symmetric $d\times d$-matrix of variables over $\mathbb Z[x_{ij}]$ for $1\leq i < j \leq d$, then the ideal $J$ generated by the $(2d+4)\times (2d+4)$-pfaffians of $N$ is generically perfect (cf.\  \cite{kleppe-laksov} or \cite[p.\ 53]{dep} and \cite[Prop.\ 4.1]{bruns}).  Furthermore,~\cite[Thms.\ 3.9 and 3.13]{bruns} show that the same statement holds if we specialize the entries of the matrix to a regular sequence.  Finally, ~\cite[Thm.\ 3.1]{brennan} shows that the entries of the matrix $BC-CB$ are a regular sequence on $ k[b_{ij},c_{ij}]$; it follows that $X$ is an integral subscheme of codimension $\binom{d-2}{2}$.

Let $p'$ be the restriction of $p$ to $X\times GL(S_1)$.  We claim that the map $p': X\times GL(S_1)\dashrightarrow \Gr(e,S_2^*)$ surjects onto the set $\Xi(d,2d+2,d)-P$.  To see this, note that by performing row and column operations on the matrix $\psi_1(Id,B,C)$, it follows that $\kappa_1(Id, B,C)\leq 2d+2$ if and only if the rank of $BC-CB\leq 2$.  This shows that $\Xi(d,2d+2,d)-P$ is irreducible.

By semicontinuity, the dimension of a general fiber of $p'$ is at least the dimension of a general fiber of $p$.  Hence:
\[
\dim(X\times GL(S_1))-\dim \left( \Xi(d,2d+2,d) \setminus P \right) \geq \dim(\mathcal T\times GL(S_1))-\dim(\Gr(e,S_2^*)),
\]
and it follows that
\[
\codim(\Xi(d,2d+2,d)-P,\Gr(e,S_2^*))\geq \codim(X,\mathcal T)=\binom{d-2}{2}
\]
On the other hand, since $\Xi(d,2d+2,d)$ is locally cut out by the $(2d+4)\times (2d+4)$-Pfaffians of a $3d\times 3d$ matrix, the codimension is at most $\binom{d-2}{2}$.  Hence $\Xi(d,2d+2,d)-P$ is an irreducible subset of codimension $\binom{d-2}{2}$, as claimed.
\end{proof}

We now wish to extend the result of the previous lemma from the open set $\Gr(3,S_2^*)\setminus P$ to the whole Grassmanian.  We do this by showing that the codimension of $P$ is sufficiently large.

Let $A=(a_{ij})$, $B=(b_{ij})$ and $C=(c_{ij})$ be three $d\times d$ matrices of indeterminates. Let $u,v,w$ be new indeterminates and let $M:=uA+vB+wC$.  If we specialize $A,B,$ and $C$ to be symmetric matrices, then coefficients in $k[a_{ij}, b_{ij}, c_{ij}]$ of the determinant of $M$ define an ideal $L$ which cuts out the preimage of $P$ under the rational map $\Spec k[a_{ij}, b_{ij}, c_{ij}]\dashrightarrow \Gr(3,S_2^*)$.

In order to produce the desired upper bound for the dimension of $V(L)$, we choose a monomial ordering $\preceq$ and find an ideal $L'\subseteq \initial_{\preceq}(L)$ of high codimension.  We introduce some notation.  Let $\preceq$  be the revlex order determined by any total ordering on the variables such that $c_{ij}\prec b_{k,l} \prec a_{m,n}$ and such that $i+j>k+s$ implies $h_{i,j}\preceq h_{k,s}$ for $h\in\{a,b,c\}$.  Let $\alpha,\beta,\gamma$ be nonnegative integers such that $\alpha+\beta+\gamma=d$.  Let $f_{\alpha,\beta,\gamma}\in k[a_{ij}, b_{ij}, c_{ij}]$ be the coefficient of the monomial $u^{\alpha}v^{\beta}w^{\gamma}$ in the determinant  $D:=\det(M)$.

We say that a sequence $\mathcal S\in \{0,1,2,3\}^d$ is \defi{of type } $(\alpha,\beta,\gamma)$ if it contains $\alpha$ 1's, $\beta$ 2's, and $\gamma$ 3's.  Given a sequence $\mathcal S$ of type $(\alpha,\beta,\gamma)$ we build a $d\times d$ matrix $M_{\mathcal S}$ whose $i$-th column is the $i$-th column of $uA,vB,wC$ or $M$ depending on whether $\mathcal{S}_i$ is $1,2,3$ or $0$ respectively.

\begin{lemma}\label{lem:grobner}
With notation as above, we have:
\begin{enumerate}
\item Among all monomials appearing in the $r\times r$ minors of $A$, the unique $\preceq$-maximal monomial is $m^*:=\prod_{i=1}^{r} a_{i,r+1-i}$.
\item For nonnegative integers $\alpha,\beta,\gamma$,
\[
\frac{\partial^{\alpha+\beta+\gamma}D}{\partial u^\alpha \partial v^\beta \partial w^\gamma} = \sum_{\mathcal S \text{ of type } (\alpha,\beta,\gamma)} \det (M_{\mathcal S}).
\]
\item{ For nonnegative integers $\alpha,\beta,\gamma$ such that $\alpha+\beta+\gamma=d$, the unique maximal monomial appearing in $f_{\alpha,\beta,\gamma}$ is
\[\prod_{i=1}^\gamma c_{i,\gamma+1-i} \prod_{j=1}^\beta b_{\gamma+j,\beta+\gamma+1-j} \prod_{k=1}^\alpha a_{\beta+\gamma+k,\alpha+\beta+\gamma+1-k}.\]}

\end{enumerate}
\end{lemma}
\begin{proof}
(1)  A monomial $m$ appears in an $r\times r$ minor of $A$ if and only if there exist subsets $I,I'\subseteq \{1, \dots, d\}$ of cardinality $r$, and a bijection $\sigma\colon I\to I'$, such that $m=\prod_{i\in I} a_{i,\sigma(i)}$.  Among such monomials $m$, we will have $m\prec m^*$ if $m$ contains at least one variable of the form $a_{i,\sigma(i)}$ with $i+\sigma(i)>r+1$.  For every such monomial $m$ we have: $\sum_{i\in I} i+\sigma(i)=\sum_{i\in I} i+\sum_{i'\in I'} i'\geq r(r+1)$.  Hence, if $I\ne\{1, \dots, r\}$ or $I'\ne \{1, \dots, r\}$, then the previous inequality is strict.  In this case, $m$ contains a variable $a_{i,j}$ with $i+j>r+1$.  If on the other hand $I=I'=\{1,\dots, r\}$ and $i+\sigma(i)\leq r+1$ for all $i$, then the bijection $\sigma$ must be $\sigma(i)=r+1-i$.

(2) follows by induction on $\alpha+\beta+\gamma$ by the well known fact that any partial derivative of the determinant of a matrix can be expressed as a sum of determinants of the matrices obtained by taking partial derivatives of the columns one at a time.

(3)  From part (2), it follows that every monomial appearing in $f_{\alpha,\beta,\gamma}$ can be written as
\[
m=\prod_{i\in I_1} a_{i,\sigma(i)}\prod_{j\in I_2} b_{j,\sigma(j)}\prod_{k\in I_3} c_{k,\sigma(k)}
\]
where $I_1, I_2, I_3$ is a set partition of $\{1, \dots, d\}$ with cardinalitites $\alpha, \beta, \gamma$, and where $\sigma$ is a permutation in $S_d$.  Since $\preceq$ is reverse lexicographic, we can maximize parts $c,b$ and $a$ independently and in that order.  The statement then follows by part (1).
\end{proof}

\begin{cor}\label{cor:genwaldoCor}
The $\kappa$-cycle $\Xi(d,2d+2,d)$ is irreducible of codimension $\binom{d-2}{2}$ when $4\leq d\leq 8$.
\end{cor}
\begin{proof}
Since $\Xi(d,2d+2,d)$ is cut out by an ideal generated by the $2d+4\times 2d+4$-Pfaffians of a $3d\times 3d$ skew-symmetric matrix, we have that every component of $\Xi(d,2d+2,d)$ has codimension at most $\binom{d-2}{2}$.  If we can show that $\codim(P,\Gr(3,S_2^*))>\binom{d-2}{2}$, then it will follow from Proposition~\ref{prop:genwaldoKcycles} that $\Xi(d,2d+d,d)$ is irreducible.  Consider the rational map
\[ p: \Spec k[a_{ij}, b_{ij}, c_{ij}] \dashrightarrow \Gr(3,S_2^*)
\]
which sends a triple of symmetric matrices to their span.  Since the fibers of $p$ have constant dimension, we have that $\codim(\overline{p^{-1}(P)})$ equals the codimension of $P$ in the grassmanian.  Specializing Lemma~\ref{lem:grobner} to the case of symmetric matrices, we obtain explicit formulas for producing monomials in the $\preceq$-initial ideal of the ideal defining $\overline{p^{-1}(P)}$.  Implementing these formulas in Macaulay2 yields the following lower bounds for the codimension of $P$:
\[
\begin{array}{|c|c|c|}
\hline
d &\binom{d-2}{2} &\codim(P, \Gr(3,S_2^*))   \\\hline
4 & 1 & \geq 9\\ \hline
5 & 3 & \geq 11 \\ \hline
6 & 6& \geq 13\\ \hline
7 & 10 &\geq 15\\ \hline
8 & 15 &\geq 17\\ \hline
\end{array}
\]
\end{proof}

%
%\begin{remark}
%The above computations show that, when $4\leq d\leq 8$ and $e=3$, the $(2d+4)\times (2d+4)$ Pfaffians defining the condition $\kappa_1\leq 2d+2$ generate a prime ideal sheaf on $\Gr(3,S_2^*)$.  Hence, if the scheme theoretic intersection of $\Gr(3,S_2^*)$ and the smoothable component inside of the Hilbert scheme is reduced, then this Pfaffian ideal would yield the correct scheme structure for the intersection.  It would be interesting to understand whether or not this is the case.
%\end{remark}

%%%%%%%%%%%%%%%%%%%%%%%%%%%
%%%%%%%%%%%%%%%%%%%%%%%%%%%
\section{Proofs of Theorems 1.2, 1.3, and 1.4}\label{sec:proofs}
%%%%%%%%%%%%%%%%%%%%%%%%%%%
%%%%%%%%%%%%%%%%%%%%%%%%%%%
We are now prepared to prove Theorems \ref{thm:1denonsmoothable}, \ref{thm:minexamples} and \ref{thm:genwaldo}.

\begin{proof}[Proof of Theorem~\ref{thm:1denonsmoothable}]
Let $I'$ be a generic smoothable $(1,d,e)$-ideal.  Proposition~\ref{prop:genKvec} part~\eqref{prop:genKvec:2} implies that $\kappa(I')$ satisfies conditions \eqref{eqn:nec:1} and \eqref{eqn:nec:j}.  Since the $\kappa$-vector is lower semicontinuous on $\Gr(e,S_2^*)$, it follows that these conditions are necessary for the smoothability of $I$.
\end{proof}

\begin{example}
The criteria of Theorem~\ref{thm:1denonsmoothable} allows for the explicit construction of nonsmoothable ideals, and the proof of Proposition~\ref{prop:genKvec} suggests a method for constructing examples.  For instance, let $d=15$ and let $q_1=\sum_{i=1}^{15} y_i^2, q_2=\sum_{i=1}^{15} iy_i^2$ and $q_3=\sum_{i=1}^{7} y_iy_{15-i}$.  Let $I$ be the $(1,15,3)$-ideal with $I_2^\perp=\langle q_1, q_2, q_3\rangle$.  Then $\kappa_1(I)=44$ and thus $I$ is nonsmoothable.
\end{example}
\begin{remark}
The bounds from the above theorem for $\kappa_1$ of a smoothable ideal give a partial response to Problem 18.40 of~\cite{cca}.  In particular, let $U\subseteq \Gr(e,S_2^*)$ be some open affine defined by inverting one of the Pl\"ucker coordinates.  Then we may define a map of free modules:
\[
\Psi_1: \wedge^1 (\mathcal O_U)^e \to \wedge^2 (\mathcal O_U)^e
\]
which specializes to $\psi_1(I)$ for any $I\in U$.  Let $f=(e-1)d+\binom{e}{2}$ and let $F$ be any $(f+1)\times (f+1)$-minor of $\Psi_1$.  Note that $F$ vanishes on $R^d_{1+d+e}\cap U$ since $\kappa_1(I)\leq f$ for any $I\in R^d_{1+d+e}\cap U$ by Proposition~\ref{prop:genKvec} part~\eqref{prop:genKvec:2}.  Let $g$ the rational map $(\mathbb A^d)^n\dashrightarrow \Gr(e,S_2^*)$ as in Lemma~\ref{lem:M2}.  The pullback $g^*(F)$ then induces an algebraic relation among the determinants $\Delta_\lambda$ for each $F$.  It would be interesting to give a more invariant description of these relations among the $\Delta_\lambda$, and to give a combinatorial proof of the corresponding algebraic identities.
\end{remark}

\begin{proof}[Proof of Theorem~\ref{thm:minexamples}]
Let $I$ define a minimal degree subscheme of $\mathbb A^d$ which is not smoothable and which cannot be embedded in $\mathbb A^{d-1}$.  We may assume that $S/I$ is local.  If the degree of $I$ is strictly less than $d+3$, then the Hilbert function of its associated graded ring is either $(1,d), (1,d,1), (1,d,1,1)$ or $(1,d,2)$.  Propositions 4.12 and 4.13 of \cite{cartwright} show that all such ideals are smoothable.  Now let $I$ have degree $d+4$.  If the Hilbert function of the associated graded ring of $I$ is not $(1,d,3)$, then it must be either $(1,d,1,1,1)$ or $(1,d,2,1)$.  Propositions 4.12, 4.14 and 4.15 of \cite{cartwright} show that all such ideals are smoothable as well.  Hence it only remains to consider ideals whose associated graded ring has Hilbert function $(1,d,3)$.  Every such ideal is homogeneous.  Theorem~\ref{thm:1denonsmoothable} implies that a generic $(1,d,3)$-ideal is not smoothable for $d\geq 4$.
\end{proof}

\begin{proof}[Proof of Theorem~\ref{thm:genwaldo}]
Let $\mathcal{Z}$ be the locus of smoothable $(1,d,3)$-ideals.  By Theorem~\ref{thm:1denonsmoothable}, $\mathcal{Z}\subseteq \Xi(d,2d+2,d)$. By Lemma~\ref{lem:M2}, when $d\leq 11$ the set $\mathcal{Z}$ of smoothable $(1,d,3)$-ideals has codimension at most ${d-2}\choose{2}$ in $\Gr(3,S_2^*)$. Hence, for $d\leq 11$, the equality $\mathcal{Z}=\Xi(d,2d+2,d)$ holds if this $\kappa$-cycle is irreducible and has codimension ${d-2}\choose{2}$ in $\Gr(3,S_2^*)$. Part (1) then follows by Corollary~\ref{cor:genwaldoCor} and part (2) follows from Proposition~\ref{prop:genwaldoKcycles}.

For the last statement of the theorem, we let $r: \mathbb A^d\times\Gr(3,S_2^*)\to H^d_{d+4}$ act by translations.  Restricting the map $r$ to $\mathbb A^d\times \mathcal Z$ gives an injection into $R^d_{d+4}$, which is not surjective.  Hence the dimension of $\mathcal Z$ is strictly less than $d(d+4)-d=d^2+3d$.  On the other hand, if $d\geq 12$ then $\dim \Xi(d,2d+2,d)\geq d^2+3d$.  It follows that $\mathcal Z \subsetneq \Xi(d,2d+2,d)$.  Thus there exist smoothable and nonsmoothable ideals with $\kappa(I')=(d,2d+2,d)$, and therefore knowledge of the $\kappa$-vector is not sufficient for deciding smoothability when $d\geq 12$.
\end{proof}
\begin{remark}
For $d=9,10,11$, we do not know if the hypothesis that $I_2^\perp$ contains a nonsingular quadric is necessary for the conclusion in part (2) of Theorem~\ref{thm:genwaldo}.
\end{remark}

%%%%%%%%%%%%%%%%%%%%%%%%%%%
%%%%%%%%%%%%%%%%%%%%%%%%%%%
\section{Examples and Proof of Proposition~\ref{prop:codim1}}\label{sec:applications}
%%%%%%%%%%%%%%%%%%%%%%%%%%%
%%%%%%%%%%%%%%%%%%%%%%%%%%%
In this section, we explain how $\kappa$-vectors were used to produce the examples from the introduction. We first show that $\kappa$-vectors provide information about deformations of $0$-schemes beyond smoothability.

We say that an ideal $J$ in $k[x_1,\dots, x_d]$ is a \defi{$(1,d',e)^{+d-d'}$-ideal} if it can be written as $J=J'\cap J''$ where $J''$ is a homogeneous ideal with Hilbert function $(1,d',e)$ and $J'$ is the ideal of a collection of $d-d'$-distinct points, none of which is the origin. We refer to $J''$ as the \defi{$(1,d',e)$-component} of $J$.

%For $d'\leq d$, let $J''\subseteq k[x_1, \dots, x_d]$ be a homogeneous ideal with Hilbert function $(1,d',e)$ and such that $J''$ contains $x_{d'+1}, \dots, x_d$.  Then there exists a unique $(1,d',e)$-ideal $K''\subseteq k[x_1, \dots, x_{d'}]$ such that
%\[
%J''=i(K'')+(x_{d'+1}, \dots, x_d)
%\]
%where $i$ is the natural inclusion.  Observe that $\kappa(J'')=\kappa(K'')$.  Further, by acting with $GL(d)$, we may always assume that any homogeneous $(1,d',e)$-ideal contains $x_{d'+1}, \dots, x_d$ without changing its $\kappa$-vector.
\begin{theorem}\label{thm:hop} Let $I$ be a $(1,d,e)$-ideal in $k[x_1,\dots, x_d]$.
\begin{enumerate}
\item{If $\kappa_0(I)\leq d'$, then $I$ deforms to a $(1,d',e)^{+d-d'}$-ideal whose $(1,d',e)$-component $J''$ satisfies $\kappa(J'')=\kappa(I)$.}
%\item{If $I$ deforms to a $(1,d',e)^{+d-d'}$-ideal, then $\kappa_0(I)\leq d'$. }

\item{If $I$ deforms to a $(1,d',e)^{+d-d'}$-ideal $J$ with $(1,d',e)$-component $J''$, then $\kappa(I)\leq \kappa(J'')$. }
\end{enumerate}

\end{theorem}
\begin{proof}
Assume that $\kappa_0(I)=d_0\leq d'$. Then, up to $GL(d)$-action, $(I_2)^{\perp}$ is spanned by $e$ quadrics in $k[y_1,\dots, y_{d_0}]$. As a result, the ideal
\[J'':=(x_{d'+1},\dots, x_{d})+I\]
is a $(1,d',e)$-ideal and $\kappa(J'')=\kappa(I)$ since $J''_2=I_2$. Let $J'$ be the ideal of a collection of $d-d'$ points $p_{d'+1},\dots, p_{d}$ in $\mathbb{A}^d$ with the property that their last $d-d'$ coordinates are linearly independent, and define $J:=J'\cap J''$. Our choice of $J'$ ensures that, for any $c_{d'+1},\dots, c_{d}\in k$, there exists a linear form $\ell(x_{d'+1},\dots, x_{d})$ with $\ell(p_{d'+i})=c_{d'+i}$ for all $i$. Hence, for any $q\in I$, there exists a linear form $\ell_q$ such that $q+\ell_q\in J$. Since $J$ contains no linear form, it follows that $\initial_{(1,\dots, 1)}(J)\subseteq I$. Both ideals coincide since they have the same colength.  Thus $I$ deforms to $J$, proving (1).

For part (2), let $Z$ be the subset of the Hilbert scheme consisting of $(1,d',e)^{+d-d'}$-ideals $Q$ whose $(1,d',e)$-component $Q''$ satisfies $\kappa(Q'')\leq \kappa(J'')$. By assumption, $I\in \overline{Z}$. Consider the map $\pi: Z\dashrightarrow \Gr(e,S_2^{*})$ given by $J\mapsto \initial_{(1,\dots,1)}(J)$.
If $Q\in Z$ does not contain a linear form, then we have $\kappa(\pi(Q))=\kappa(Q'')$. Hence $\kappa(\pi(Q))\leq \kappa(J'')$ for all points in $\pi(Z)$.  Since $I\in \overline{\pi(Z)}$, it follows from semicontinuity of the $\kappa$-vector that $\kappa(I)\leq \kappa(J'')$.
\end{proof}

\begin{proof}[Proof of Proposition~\ref{prop:codim1}]
We first prove the proposition in the case that $n=15$ and $d=11$.  Hence, we must produce a subset $Z\subseteq R^{11}_{15}$ of codimension $1$ such that every point of $Z$ is singular in the Hilbert scheme. Consider the action $r: \mathbb{A}^{11}\times H^{11}_{15}\rightarrow H^{11}_{15}$ by translations and let $Y$ be the image of  $\Gr(3,S_2^*)\times \mathbb A^{11}$ in $H^{11}_{15}$ under this action. Note that $Y$ is not contained in the smoothable component.  We define $Z:=R^{11}_{15}\cap Y$. By Lemma~\ref{lem:M2}, it follows that
$$
\dim Z=153 +\dim \mathbb A^{11}=164=\dim R^{11}_{15}-1.
$$
Since $Z$ belongs to the intersection of two components of $H^{11}_{15}$, it follows that every point of $Z$ is singular in the Hilbert scheme.  Note also that every element of $Z$ has a deformation to a smooth $0$-scheme and has a deformation to a nonsmoothable $0$-scheme.

We now let $n\geq 15$ and $d\geq 11$ and define the family $\widetilde{Z}_{n,d}$ where an element of $\widetilde{Z}_{n,d}$ corresponds to a $0$-scheme $\Gamma=\Gamma_1\sqcup \Gamma_2\subseteq \mathbb A^d$ of degree $n$ such that $\Gamma_1$ is abstractly isomorphic to some element of $Z$ and such that $\Gamma_2$ is the union of $n-15$ distinct points.  By \cite[p.\ 4]{artin}, it follows that any abstract deformation of a $0$-scheme can be a lifted to an embedded deformation of the $0$-scheme in any embedding.  We may then conclude that every element of $\widetilde{Z}_{n,d}$ has a deformation to a smooth $0$-scheme and has a deformation to a nonsmoothable $0$-scheme.  In particular, $\widetilde{Z}_{n,d}$ lies on the smoothable component and at least one other component of the Hilbert scheme, and hence every point in $\widetilde{Z}_{n,d}$ is singular in the Hilbert scheme.

It remains to show that $\dim \widetilde{Z}_{n,d}=\dim R^d_n-1=nd-1$.  First, we choose an $11-$dimensional vector subspace $V\subseteq \langle x_1, \dots, x_d\rangle$, and then we choose an ideal $I\subseteq \Sym(V)\cong k[z_1, \dots, z_{12}]$ such that $I$ belongs to the family $Z$ from above and such that $I$ is supported at the origin.  Next, we choose a basis $l_1, \dots, l_{d-11}$ of $\langle x_1, \dots, x_d\rangle / V$ and a basis $p_1, p_2, p_3$ of $\Sym_2(V)/I_2$, and we choose parameters $\lambda_{i,j}\in k$ for $i=1, \dots, d-11$ and $j=1,2,3$.  By considering the ideal generated by $I+\langle l_i+\sum_{j=1}^3 \lambda_{i,j}p_j\rangle$, we have parametrized the possible choices for $\Gamma_1$ supported at the origin.  The dimension of this family is
\[
11(d-11)+\left( \dim(Z)-11\right) + 3(d-11)=14d-1
\]
To parametrize all possible choices for $\Gamma$, may also translate $\Gamma_1$ anywhere in $\mathbb A^d$, and we may choose any generic $n-15$ points in $\mathbb A^d$ for $\Gamma_2$.  This yields:
\[
\dim \widetilde{Z}_{n,d}=(14d-1)+d+(n-15)d=nd-1
\]
as desired.
\end{proof}
%Since the codimension one locus $\widetilde{Z}_{n,d}$ is contracted under the Chow morphism $\pi: H^d_n \to \Sym_n(\mathbb A^d)$, it might be interesting to consider $\widetilde{Z}_{n,d}$ from the perspective of birational geometry.

\begin{proof}[Example~\ref{examples} part~\eqref{ex:family}]
Let $I$ be a generic $(1,d,e)$-ideal and let $I'$ be a generic smoothable $(1,d,e)$-ideal.  If $e>3$ and $d> \binom{e}{2}$, then by Proposition~\ref{prop:genKvec} part~\eqref{prop:genKvec:1}, we conclude that
\[
\kappa_1(I')\leq (e-1)d+\binom{e}{2} <ed=\kappa_1(I).
\]
Since $\kappa_1$ is lower-semicontinuous, it follows that whenever $\kappa_1(I')>(e-1)d+\binom{e}{2}$, the ideal $I'$ is not smoothable.  For $e=3$, there are two cases to consider.  If $d$ is even, then $\kappa_1(I')=2d+3<3d=\kappa_1(I)$ since $d\geq 4$.  If $d$ is odd, then $\kappa_1(I')=2d+3<3d-1=\kappa_1(I)$ since $d\geq 5$.
\end{proof}

\begin{proof}[Example \ref{examples} part \eqref{ex:intersect}]

We wish to show that the Hilbert scheme of $11$ points in $\mathbb A^7$ has two components whose intersection is {\em not} contained in the smoothable component.  Let $Y_1$ be the irreducible component of $H^7_{11}$ containing the set of $(1,6,3)^{+1}$ ideals. We claim that the dimension of $Y_1$ is at most $77$ and that $Y_1$ is not the smoothable component.

Consider the $(1,6,3)$-ideal $J''\subseteq \mathbb Q[x_1, \dots, x_6]$ defined by
\[
(J''_2)^\perp =\langle y_1^2+y_2^2+\dots +y_6^2, y_1^2+2y_2^2+3y_3^2+5y_4^2+7y_5^2+11y_6^2, y_1y_6+y_2y_5+y_3y_4 \rangle.
\]
Let $J'$ be the ideal of the point $(0,0,0,0,0,0,1)$ in $\mathbb A^7$ and let $J:= \left( J''+(x_7)\right) \cap J'$ in $\mathbb Q[x_1, \dots, x_7]$.  Note that $J$ is a $(1,6,3)^{+1}$-ideal.  Using Macaulay2~\cite{M2}, we compute that the tangent space dimension of $J$ is $77$.  Hence, the component $Y_1$  has dimension at most $77$.  Since $\kappa(J'')=(6,18,6)$, it follows from Theorem~\ref{thm:1denonsmoothable} and \cite[p.\ 4]{artin} that $J$ is not smoothable, and thus that $Y_1$ is not the radical component.

Let $Y_2$ be the irreducible component containing the set of $(1,7,3)$-ideals.   We claim that $Y_2$ is neither $Y_1$ nor the smoothable component.  This follows immediately from the fact that
\[
\dim Y_2 \geq \dim \Gr(3,S_2^*)\times \mathbb A^7=82>77=\dim R^{7}_{11}\geq \dim Y_1.
\]

We now show that the ideal $I:=\initial_{(1,\dots,1)}(J)$ is a nonsmoothable ideal which belongs to $Y_1\cap Y_2$. Note that $I$ belongs to $Y_1$ because it is a degeneration of $J$, and that $I$ belongs to $Y_2$ since it is a $(1,7,3)$-ideal.  The proof of Theorem~\ref{thm:hop} implies that $\kappa(I)=\kappa(J'')=(6,18,6)$.  Since $18>2\cdot 7+2$, Theorem~\ref{thm:1denonsmoothable} implies that $I$ is not smoothable.
\end{proof}

\begin{proof}[Example~\ref{examples} part \eqref{ex:10pts}]
We must show that some deformations of $I$ are determined by the $\kappa$-vector according to the following table:

\begin{center}
\begin{tabular}{| l | c | }\hline
$I$ deforms into a \dots & if and only if \dots \\ \hline
union of $9$ points &$\kappa(I)\leq (5,12,5)$\\ \hline
$(1,4,3)^{+1}$-ideal & $\kappa(I) \leq (4,12,4)$ \\ \hline
{\em smoothable} $(1,4,3)^{+1}$-ideal & $\kappa(I)\leq (4,10,4)$\\ \hline
\end{tabular}
\end{center}

\noindent The first line of the table follows from Theorem~\ref{thm:genwaldo}.  The second line of the table follows from Theorem~\ref{thm:hop}.  We now consider the last line of the table.  By \cite[p.\ 4]{artin} and Theorem~\ref{thm:genwaldo}, a $(1,4,3)^{+1}$ ideal $J$ is smoothable if and only if its $(1,4,3)$-component $J''$ has $\kappa(J'')\leq (4,10,4)$.  The last line of the table then follows from Theorem~\ref{thm:hop}.
\end{proof}

\begin{remark}[Generalized $\kappa$-vector]
The notion of $\kappa$-vector can be extended to subspaces of polynomials of any degree.  In particular, let $V\in \Gr(e,S_m^*)$ be a subspace with basis $\mathbf{f}=\langle f_1, \dots, f_e\rangle$.  Define the $\kappa$-matrix $\kappa^{\text{mat}}(V):=(\kappa_{i,j}(V))$, where $\kappa_{i,j}(V)$ is the rank of the linear map:
\[ \xymatrix{
S_i\otimes \bigwedge^{j}V\ar[r]^{\wedge \mathbf{f} \quad} & S_{i-m}\otimes \bigwedge^{j+1}V\\
}
\]
and $S_{k}$ is defined to be $S_{-k}^*$ if $k<0$.  It would be interesting to know if this generalized numerical invariant induces further nontrivial obstructions for deformations of homogeneous ideals.
\end{remark}
%%%%%%%%%%%%%%%%%%%%%%%%%%%
%%%%%%%%%%%%%%%%%%%%%%%%%%%
\section*{Acknowledgements}  We thank Dustin Cartwright and Bianca
Viray for useful conversations, and for many suggestions which influenced this paper.  In addition, we thank Bernd Sturmfels for stimulating our interest in the subject, and we thank David Eisenbud for his support throughout our work on this project.  We also thank Mats Boij for an illuminating conversation about inverse systems. We thank Anthony Iarrobino and Kyungyong Lee for comments on an earlier draft.  We thank Dan Grayson and Mike Stillman, the makers of Macaulay2, which was very useful at all stages of our work on this project.  Finally, we thank the reviewer for a number of helpful suggestions.
%%%%%%%%%%%%%%%%%%%%%%%%%%%
%%%%%%%%%%%%%%%%%%%%%%%%%%%

\end{document}